\documentclass[a4paper,11pt,notitlepage]{amsart}
\usepackage{amsmath,amssymb,mathrsfs,amsthm}
\usepackage{hyperref}
\usepackage{verbatim}
\usepackage[spanish,USenglish]{babel} 
\usepackage{color}
\usepackage{graphicx}

\usepackage[
  margin=30mm,
  marginparwidth=25mm,     
  marginparsep=2mm,       
  bottom=25mm,
  ]{geometry}

\newtheorem{theorem}{Theorem}

\newtheorem{lemma}[theorem]{Lemma}

\newtheorem{remark}[theorem]{\it Remark}

\newcommand{\edproof}{ $\hfill {\Box}$}

\newcommand\Real{{\mathfrak R}{\mathfrak e}\,} %

\newcommand{\R}{\mathbb{{R}}}
\newcommand{\Z}{\mathbb{{Z}}}
\newcommand{\N}{\mathbb{{N}}}

\newcommand{\C}{\mathbb{{C}}}

\begin{document}

\title[Discrete H\"older spaces]{Discrete H\"older spaces,  their characterization via semigroups associated to the discrete Laplacian and kernels estimates.}

\author[L. Abadias]{Luciano Abadias}
\address{\newline Luciano Abadias\newline Departamento de Matem\'aticas, Instituto Universitario de Matem\'aticas y Aplicaciones, Universidad de Zaragoza, 50009 Zaragoza, Spain.}
\email{labadias@unizar.es}

\author[M. De Le\'on-Contreras]{Marta De Le\'on-Contreras}
\address{\newline
       Marta De Le\'on-Contreras \newline
       Department of  Mathematics and Statistics,  
       University of Reading, \newline
	   Reading RG6 6AX, United Kingdom}
\email{m.deleoncontreras@reading.ac.uk }

\thanks{The first named author have been partly supported by Project PID2019-105979GB-I00 of the MICINN of Spain, Project E26-17R, D.G. Arag\'on, Universidad de Zaragoza, Spain, and Project for Young Researchers JIUZ-2019-CIE-01 of Fundaci\'on Ibercaja and Universidad de Zaragoza, Spain, and the second author has been partially supported by EPSRC Research Grant EP/S029486/1.}



\begin{abstract}
In this paper we characterize  the discrete H\"older spaces by means of the heat and Poisson semigroups associated to the discrete Laplacian. These characterizations allow us to get regularity properties of fractional powers of the discrete Laplacian and the Bessel potentials along these spaces and also in the discrete Zygmund spaces in a more direct way than using the pointwise definition of the spaces.

To obtain our results, it has been crucial to get boundedness properties of the heat and Poisson kernels and their derivatives in both space and time variables. We believe that these estimates are also of independent interest.
\end{abstract}

\date{}

\maketitle

\section{Introduction}
\setcounter{theorem}{0}
\setcounter{equation}{0}

Classical H\"older spaces $C^\alpha(\R^n), \, \alpha>0$, $\alpha\not\in\N$ (also denoted by $C^{k,\beta}(\R^n)$ or $C^{k+\beta}(\R^n)$, being $k+\beta=\alpha,$ $k\in\N_0,$ and $0<\beta<1$) are classes of smooth functions that are very important in partial differential equations, harmonic analysis and function theory. When $0<\alpha<1$ they are defined as the set of (bounded) functions 
$f$  such that  \begin{equation}\label{liphol} |f(x+z)-f(x)| \le C |z|^\alpha\:\:x,z\in\R^n.\end{equation} These spaces are in between of  the space of bounded continuous functions, $\mathcal{C}^0(\R^n)$, and the one of bounded differentiable functions with bounded continuous derivative, $\mathcal{C}^1(\R^n).$  These  spaces are  usually called either Lipschitz or H\"older classes. For $\alpha =1$ the natural space was introduced by Zygmund  \cite[Chapter II]{Zygmund} and it is the set of continuous and bounded functions  $f$ such that $$ |f(x+z)+f(x-z)-2f(x)  | \le C |z|, \:\; x,z\in\R^n.$$ This  space is commonly  known as the Zygmund's space and we shall denote it  by $Z$. It can be shown that if we denote  by ${\rm Lip}$ the space of functions satisfying (\ref{liphol}) for  $\alpha =1$, then 
$\mathcal{C}^1(\R^n) {\subsetneq}\; {\rm Lip}  {\subsetneq}\; Z$, see \cite{Krantz1}.
 Given $\alpha>1$, $C^{\alpha}(\R^n) $ is the set of functions such that all {the} derivatives {of} order less or equal than $[\alpha]$ are continuous and  bounded and the derivatives of order  $[\alpha]$ belong to $C^{\alpha-[\alpha]}(\R^n)$.

In  the 60's of last century, Stein and Taibleson, see \cite{Stein}  and \cite{Tai1,Tai2,Tai3}, characterized bounded H\"older functions via some integral estimates of the Poisson semigroup, $\{e^{-y\sqrt{-\Delta}}\}_{y>0},$ and of  the Gauss semigroup, $\{e^{\tau{\Delta}}\}_{\tau>0}$. 	 The advantage of this kind of results is that the semigroup descriptions allow to obtain regularity results in these spaces in a more direct way, avoiding the long, tedious and sometimes cumbersome computations that are needed when the pointwise expressions are handled.
The works of Taibleson and Stein raise the question of analyzing  H\"older spaces adapted to different  ``Laplacians'' and to find  their pointwise
and semigroup characterizations. 

{In \cite{GU}, Lipschitz spaces adapted to the Ornstein-Ulhenbeck operator, $\mathcal{O}=-\frac{1}{2}\Delta+x\cdot\nabla$,  were defined by means of its Poisson semigroup,  $\{e^{-y\sqrt{\mathcal{O}}}\}_{y>0}$, and in \cite{LS} a pointwise characterization was obtained for $0<\alpha<1$. }

In the case of Schr\"odinger operators, $ -\Delta+ V$  on 
 $\mathbb{R}^n,$ $ n \ge 3$,
where $V$  satisfies a reverse H\"older inequality for some $q> n/2$, adapted Lipschitz classes were pointwise defined  in \cite{BH} for  $0<\alpha<1$. In \cite{MSTZ} the authors characterized these spaces by means of the Poisson semigroup $\{e^{-y\sqrt{-\Delta+ V}}\}_{y>0}$ and they got  boundedness  of fractional powers of $-\Delta+ V$ 
 in these spaces for  $0<\alpha<1$. Recently, in \cite{DL-CT2} it was extended the pointwise and semigroup (heat and Poisson) characterizations  to the range $0<\alpha\le 2-n/q$.   In addition, the authors used those semigroups definitions to get regularity results regarding fractional operators related to $-\Delta+V$.  Moreover, in the particular case of the Hermite operator, $-\Delta+|x|^2$, in \cite{DL-CT1} the authors got, for every $\alpha>0,$ a characterization  by means of the heat and Poisson semigroups of the adapted H\"older spaces defined in \cite{ST2} and also of the adapted parabolic  H\"older spaces introduced in \cite{DL-CT1}.

 Regarding the heat operator, $\partial_t - \Delta$,  in \cite{ST3} the parabolic H\"older spaces introduced by Krylov, see  \cite{Krylovbook},  were  characterized by means of the  Poisson semigroup $\{e^{-y\sqrt{\partial_t - \Delta}}\}_{y>0}$ and the authors used this semigroup characterization to show  regularity properties for fractional powers  $(\partial_t -\Delta_x )^{\pm \sigma}$. 
 
 In \cite{BB} it is proved that, in a general metric measure
space $(M, d, \mu)$  where $\mu$ is doubling, if $L$ is an operator  such that the heat semigroup $\{e^{tL}\}_{t>0}$ is conservative, i.e, $e^{tL}1=1$, and the associated heat kernel satisfies Gaussian bounds and a Lipschitz condition in the spatial variable, then the H\"older spaces adapted to $L$ (defined by increments) can be characterized by means of the heat semigroup, for $0<\alpha<1.$

In this paper we shall deal with the discrete H\"older spaces, $C^\alpha(\Z)$, $\alpha>0$, $\alpha\not\in\N$, whose definition we are going to recall in the following lines, and also we will introduce new discrete Zygmund classes, $Z_\alpha,$ $\alpha\in\N$. Our first aim is to prove semigroup characterizations of these spaces by using the heat and Poisson semigroups associated to the discrete Laplacian, $-\Delta_d.$ 
The heat kernel associated to the discrete Laplacian neither has a Gaussian control  nor satisfies a
 Lipschitz condition as in \cite{BB},  see \cite{GKP,Pang}. Therefore, our results are not covered by the ones in \cite{BB} and the kernels have not the same kind of good estimates and homogeneity properties than in the works in the literature we cited above. These are the first main difficulties we have faced in this problem and we have been able to sort them out by obtaining new estimates for the kernels and their derivatives, see Section \ref{semigroups}.

For $f:\Z\rightarrow\R,$ consider the discrete derivatives ``from the right'' and ``from the left'',
$$
\delta_{right}f(n):=f(n)-f(n+1), \quad\quad \delta_{left}f(n):=f(n)-f(n-1).
$$
Observe that $\delta_{right}\delta_{left}f=\delta_{left}\delta_{right}f$ and this implies that every combination of these operators is not affected by the order when they are applied. For more properties of these operators see \cite{ADT}.

Now we recall the definition of discrete H\"older spaces introduced in \cite{CRSTV}. Let   $f:\Z\to \R$ be a  sequence. For $0<\alpha<1$ we say that $f\in C^{\alpha}(\Z)$ if $$\sup_{n\neq m}\frac{|f(n)-f(m)|}{|n-m|^{\alpha}}<\infty.$$ 
For $1<\alpha<2$, $f\in C^{\alpha}(\Z)$ if
$$
\sup_{n\neq m}\frac{|\delta_{right/left}f(n)-\delta_{right/left}f(m)|}{|n-m|^{\alpha-1}}<\infty.
$$
In general, for $\alpha=k+\beta>0$, where $k\in\N_0:=\N\cup\{0\}$ and $0<\beta<1,$ we say that 
$f\in C^{\alpha}(\Z)$ if for each $l,s\in\N_0$ such that $l+s=k$
$$
\sup_{n\neq m}\frac{|\delta_{right/left}^{l,s}f(n)-\delta_{right/left}^{l,s}f(m)|}{|n-m|^{\beta}}<\infty,
$$
where $\delta_{right/left}^{l,s}:=\delta_{right}^{l}\delta_{left}^{s}$ (or any other combination of these operators such that in the end you apply $l $ times $\delta_{right}$ and $s$ times $\delta_{left}$).

Observe that $\ell^\infty(\Z)$ functions are trivially in $C^\alpha(\Z).$ Therefore, we are not going to consider bounded functions. However, we will prove, see Lemma \ref{sizeCalpha}, that for every $f\in C^\alpha(\Z)$, $\alpha>0,$ there exists $C>0$ such that
$$|f(n)|\le C(1+|n|^\alpha), \quad n\in\Z.
$$

Moreover, when $\alpha\in\N,$ we define the {\it discrete Zygmund classes}, $Z_\alpha$,  as
\begin{align*}
&Z_1:=\Bigg\{f:\Z\to \R: \:\frac{f}{1+|\cdot|}\in\ell^\infty(\Z) \quad  \text{ and }\: \sup_{n\neq 0}\frac{\|f(\cdot+n)+f(\cdot-n)-2f(\cdot)\|_\infty}{|n|}<\infty\Bigg\}\\ 
&\text{ and for  }\alpha\in\N\setminus\{1\},\\
&Z_\alpha:=\Big\{f:\Z\to \R: \: \frac{f}{1+|\cdot|^\alpha}\in\ell^\infty(\Z) \quad \text{ and   } \delta_{right/left}^{l,s}f\in Z_1, \: \; l+s=\alpha-1.\Big\}
\end{align*}
Both definitions of $C^\alpha(\Z)$ and $Z_\alpha$ involve pointwise estimates of the functions. Our first aim will be to get  their characterizations by means of semigroups.

Let $\Delta_d$ denote the discrete Laplacian on $\Z,$ that is, for each $f:\Z\to\R,$
$$(\Delta_d  f)(n):=f(n+1)-2f(n)+f(n-1),\quad n\in\Z.$$
The solution of the discrete heat problem \begin{equation}\label{eq1}
\left\{\begin{array}{ll}
\partial_t u(t,n)-\Delta_d u(t,n)=0,&n\in\Z,\,t\geq 0,\\
u(0,n)=f(n),&n\in\Z,
\end{array} \right.
\end{equation}
is given by the convolution $u(t,n)=e^{t\Delta_d}f(n):=\sum_{j\in\Z}G(t,n-j)f(j)=\sum_{j\in\Z}G(t,j)f(n-j),$ where the discrete heat kernel is $$G(t,n)=e^{-2t}I_{n}(2t),\quad n\in\Z, \: t>0,$$ being $I_n$ the modified Bessel function of the first kind and order $n\in\Z,$  see  Section \ref{semigroups} for more details. 

It seems that H. Bateman in \cite{Bateman} was the first author dealing with the solution of \eqref{eq1}. Moreover, he studied a broad set of differential-difference equations (heat and wave equations), whose solutions are given in terms of special functions; the Bessel function $J_n$, the Bessel function of imaginary argument $I_n$, the Hermite polynomial $H_n$ and the exponential function. In the last years many mathematicians have been working in this discrete heat setting. For example, in \cite{I2, I}, the author studies large time behaviour for $e^{t\Delta_d}f$ in $\ell^p(\Z)$ spaces by using the semidiscrete Fourier transform. In \cite{CGRTV}, the authors do a deep harmonic analysis study of this problem. In \cite{LR}, the authors study the spectrum of $\Delta_d$ on $\ell^p(\Z)$, the associated wave problem, and holomorphic properties of $e^{z\Delta_d}f$. In \cite{Slavik} the author proves that the solution of \eqref{eq1} behaves asymptotically as the mean of the initial value, and in \cite{AG-CM} the authors study large time behaviour in $\ell^p(\Z)$ for the solutions of \eqref{eq1} with a non-homogeneous linear forcing term.

On the other hand, the solution of the discrete Poisson problem 
\begin{equation*}\label{eqPoisson}
\left\{\begin{array}{ll}
\partial_y^2 v(y,n)-\Delta_d v(y,n)=0,&n\in\Z,\,y\geq 0,\\
v(0,n)=f(n),&n\in\Z,
\end{array} \right.
\end{equation*}
is denoted by $v(y,n)=e^{-y\sqrt{-\Delta_d}}f(n),$ $y>0,$ $n\in\Z.$ Moreover, Bochner's subordination formula allows us to write
\begin{align*}
    e^{-y\sqrt{-\Delta_d}}f(n)&=\frac{y}{2\sqrt{\pi}}\int_0^\infty \frac{e^{-\frac{y^2}{4t}}}{t^{3/2}} e^{t\Delta_d}f(n)dt=\sum_{j\in\Z}\left(\frac{y}{2\sqrt{\pi}}\int_0^\infty \frac{e^{-\frac{y^2}{4t}}}{t^{3/2}} G(t,j)dt\right)f(n-j)\\
    &:=\sum_{j\in\Z}P(y,j)f(n-j), \quad y>0, \: \; n\in\Z.
\end{align*}
To the best of our knowledge, an explicit expression of the Poisson kernel $P(y,j)$ is not known. However, by using the subordination formula and our new estimates for the heat kernel, we will be able to obtain useful bounds for $P(y,j)$, see Section \ref{semigroups}.

Now we consider the following spaces associated to the discrete Laplacian. Let $\alpha>0.$ We define \begin{align*}
\Lambda^{\alpha}_H:=\Big\{ f:\Z\to\R\,:\,  \frac{f}{1+|\cdot|^\alpha}\in\ell^\infty(\Z)  \text{ and }\exists\, C_\alpha>0 \text{  such that }  \|\partial_t^k e^{t\Delta_d}f\|_{\infty}\leq C_\alpha t^{-k+\alpha/2}, \\\quad k=[\alpha/2]+1,\: t>0\Big\}.
\end{align*}
\begin{align*}
    \Lambda^{\alpha}_P:=\Big\{ f:\Z\to\C\,:\,  \sum_{j\in\Z}\frac{|f(j)|}{1+|j|^2}<\infty \text{ and }\exists\, \widetilde{C_\alpha}>0 \text{  such that }  \|\partial_y^l e^{-y\sqrt{-\Delta_d}}f\|_{\infty}\leq \widetilde{C_\alpha} y^{-l+\alpha}, \\\quad l=[\alpha]+1,\: y>0\Big\}.
\end{align*}
The condition on the functions $\frac{f}{1+|\cdot|^\alpha}\in\ell^\infty(\Z) $, $\alpha>0$, will be enough to have the heat semigroup and its derivatives well defined. However, for the case of the Poisson semigroup we need a more restrictive condition, $\sum_{j\in\Z}\frac{|f(j)|}{1+|j|^2}<\infty$, see Section \ref{semigroups} for the details.

Now we present our main results. The first main theorem we prove is the characterization, for every $\alpha>0$, of the pointwise spaces $C^\alpha(\Z)$ and $Z_\alpha$, by means of the heat and Poisson semigroups.

\begin{theorem}\label{caractodos}
\textcolor{white}{}\newline
\begin{itemize}
    \item[(A)]Let $\alpha>0.$\\ \begin{itemize}
        \item[(A.1)] If $\alpha\not\in\N,$ then $\Lambda_H^\alpha=C^\alpha(\Z)$.\\
        
           \item[(A.2)] If $\alpha\in\N,$ then $\Lambda_H^\alpha=Z_\alpha$.\\
           
    \end{itemize}
    \item[(B)]Let $f:\Z\to\R$ such that $\sum_{j\in\Z}\frac{|f(j)|}{1+|j|^2}<\infty$. \\ \begin{itemize}
        \item [(B.1)] For every $\alpha>0,$ $\alpha\not\in\N,$ $$f\in C^\alpha(\Z)\Longleftrightarrow f\in\Lambda_H^\alpha\Longleftrightarrow f\in\Lambda_P^\alpha.$$
        
        \item [(B.2)]  For every $\alpha\in \N$,
        $$f\in Z_\alpha\Longleftrightarrow f\in\Lambda_H^\alpha\Longleftrightarrow f\in\Lambda_P^\alpha.$$
    \end{itemize}
\end{itemize}

\end{theorem}
To prove previous theorem some estimates about the discrete heat and Poisson kernels  and their derivatives are crucial (see Lemmata \ref{pointbessel}, \ref{heatbounds}, \ref{kernelest}, \ref{Poissonlema}, \ref{Poissonest}). These results complement, extend and improve some of the ones obtained in \cite{AG-CM,I2,I}. We believe that ours results are also of independent interest because we give general pointwise and $\ell^1$-estimates for the difference and derivatives of any order of the heat and Poisson discrete kernels.

Once the semigroup characterization is obtained, we have been able to get regularity results for fractional operators related with $\Delta_d$, such as Bessel potentials, $(I-\Delta_d)^\beta$, $\beta>0$, and the fractional powers $(-\Delta_d)^{\pm \beta}.$
For the definitions of these operators see Section \ref{applications}.

  \begin{theorem}\label{Besselpot}Let $\alpha,\beta>0$.
   	\begin{itemize}
	\item[(i)]  If
		$f\in 		\Lambda_H^{\alpha}$, then  $(I-\Delta_d)^{-\beta/2} f \in 	\Lambda_H^{{\alpha+\beta}} $. 
		\item[(ii)] If $f\in \ell^\infty(\Z)$, then 	 $	(I-\Delta_d)^{-\beta/2}f  \in 	\Lambda_H^{{\beta}} $.
			\end{itemize}
   \end{theorem}

To define the fractional powers of $\Delta_d$ it is necessary to define auxiliary  spaces of sequences $\ell_{\beta}$, $\beta>-1/2,$ $\beta\neq 0,$
$$
\ell_{ \beta}:=\left\{u: \Z\to\R: \:\; \sum_{m\in\Z}\frac{|u(m)|}{(1+|m|)^{1+ 2\beta}}<\infty\right\},
$$
These spaces are the discrete analogue of the spaces needed in the case of the Laplacian in $\R^n,$ see \cite{Silvestre}, and they were introduced in \cite{CRSTV}  for $\beta\in (-1/2,1)$, $\beta\neq 0$ and, for sequences $f\in\ell_{\beta}$, the authors got 
a pointwise convolution expression for $(-\Delta_d)^{\beta}f.$ We consider these spaces $\ell_{\beta}$ for any $\beta>-1/2,$ {$\beta\neq 0,$} and we are able to extend the results in  \cite{CRSTV} and get the pointwise expression of $(-\Delta_d)^{\beta}f$ for $\beta\geq 1,$ see Lemma \ref{lbeta} and Remark \ref{nucleoderiv}.

The following theorems were proved in \cite[Theorems 1.5 and 1.6]{CRSTV} for positive powers $0<\beta<1$ and negative powers with $0<\beta<1/2$,  in the discrete H\"older classes.  In \cite{CRSTV}, the authors obtained their results by using the pointwise definition of the fractional powers of the Laplacian.  Our results for the positive powers covers all values of $\beta>0$ and our proofs will be more direct and systematic. Moreover, when $\alpha\in\N$ we are considering the Zygmund spaces, while in \cite{CRSTV} the authors consider other different spaces.
     
   \begin{theorem}[Schauder estimates]\label{teoSchau} Let $ \alpha>0$, $0<\beta<1/2$ and $f:\Z\to\R.$
   	\begin{itemize}
   		\item[(i)] If  $f\in 	\Lambda_H^{\alpha}\cap \ell_{-\beta}$, then $(-\Delta_d)^{-\beta} f \in 	\Lambda_H^{\alpha+2\beta} $. 
   		\item[(ii)] If $f\in \ell^\infty(\Z)\cap \ell_{-\beta}$ , then 	 $(-\Delta_d)^{-\beta} f \in 	\Lambda_H^{2\beta}  $.  
   	\end{itemize} 
   \end{theorem}
 Since the operator $-\Delta_d$ consists of second order differences, $(-\Delta_d)^mf$, $m\in\N$, is well  defined for any sequence $f$. Thus, the fractional powers of $-\Delta$ of order $\beta>1$ can be defined  as
$$
(-\Delta_d)^\beta f=(-\Delta_d)^{[\beta]}((-\Delta_d)^{\beta-[\beta]}f), \quad \text{ for } f\in \ell_{\beta-[\beta]}.
$$
However, by using formula \eqref{positivepower} we have a definition of $(-\Delta_d)^\beta f$,  $\beta>1,$ $\beta\not\in\N$, that is valid for $f\in \ell_\beta$, which is a larger class of functions than $\ell_{\beta-[\beta]}$.
\begin{theorem}[H\"older estimates] \label{teoHolder}
   	Let $\alpha,\beta>0$ such that $0< 2\beta < \alpha $ and $f:\Z\to\R.$ \newline 
   	\begin{itemize}
   	\item [(i)] If $f\in\Lambda_H^{\alpha}\cap  {\ell_{\beta}}$, then 
   	$(-\Delta_d)^\beta f\in\Lambda_H^{\alpha-2\beta}$, being $(-\Delta_d)^\beta$ defined as in  \eqref{positivepower}.  
   	\item [(ii)] If $f\in\Lambda_H^{\alpha}$  and $\beta\in\N$, then 
   	$(-\Delta_d)^\beta f$, understood as the $\beta$-times iteration  of $(-\Delta)$, belongs to $\Lambda_H^{\alpha-2\beta}$.  
   	\end{itemize}
   \end{theorem}

It is usual to consider the natural extension of the discrete H\"older classes to the mesh of step $h>0.$ Let $h>0,$ $\Z_h:=\{nh\,:\,n\in\Z\},$ $f:\Z_h\to\R.$ One can define $$\delta_{right}f(nh):=\frac{f(nh)-f((n+1)h)}{h},\quad \delta_{left}f(nh):=\frac{f(nh)-f((n-1)h)}{h}.$$ For $\alpha=k+\beta>0$, where $k\in\N_0:=\N\cup\{0\}$ and $0<\beta<1,$ we say that 
$f\in C^{\alpha}(\Z_h)$ if there is $C>0$ such that for each $l,s\in\N_0$ with $l+s=k$
$$
\sup_{n\neq m}\frac{|\delta_{right/left}^{l,s}f(nh)-\delta_{right/left}^{l,s}f(mh)|}{h^{\beta}|n-m|^{\beta}}\leq C<\infty.
$$
All our results also hold for the spaces  $C^\alpha(\Z_h)$, $h>0,$ and the  spaces defined through the heat and Poisson semigroups  associated to the discrete Laplacian  $$\Delta_df(nh):=\frac{f((n+1)h)-2f(nh)+f((n-1)h)}{h^2}.$$ However, for simplicity we present the results when $h=1.$ Moreover, doing a tedious work component to component, one can repeat the results in the multidimensional case $\Z_h^N.$

   The paper is organized as follows. 
In Section \ref{semigroups} we prove all the results concerning pointwise and norm estimates of the discrete heat and Poisson kernels and semigroups.  In Section \ref{caractodo} we prove Theorem \ref{caractodos} and all the properties related with these spaces. Finally, in Section \ref{applications} we prove the results regarding the applications, Theorems \ref{Besselpot}, \ref{teoSchau} and \ref{teoHolder}.\\

Throughout this article $C$ and $c$ always denote positive constants that can change in each occurrence.

\section{Discrete Gaussian  and Poisson semigroups}\label{semigroups}

\subsection{Bessel functions}\label{appendix}
\setcounter{theorem}{0}
\setcounter{equation}{0}
\subsubsection{Some known results}\textcolor{white}{}\newline
Along the paper, next estimates for the Euler's gamma function  will be applied in some results. Recall that for every $\alpha,z\in\C,$  $$\frac{\Gamma(z+\alpha)}{\Gamma(z)}=z^{\alpha}(1+\frac{\alpha(\alpha+1)}{2z}+O(|z|^{-2})),\quad |z|\to\infty,$$ whenever $z\neq 0,-1,-2,\ldots$ and $z\neq -\alpha,-\alpha-1,\ldots,$ see \cite[Eq.(1)]{ET}. In particular,
\begin{equation*}\label{double3}
\frac{\Gamma(z+\alpha)}{\Gamma(z)}=z^{\alpha}\left(1+O\left({\frac{1}{|z|}}\right)\right), \quad z\in\C_+,\,\Real\alpha>0.
\end{equation*}

We denote by $I_n$ the modified Bessel function of the first kind and order $n\in\Z,$ given by \begin{equation*}\label{DefBessel}I_n(t)=\sum_{m=0}^{\infty}\frac{1}{m!\Gamma(m+n+1)}\biggl(\frac{t}{2}\biggr)^{2m+n},\quad n\in\N_0,\, t\in\C,\end{equation*} and $I_{-n}=I_n$ for $n\in\N.$

Now we give some known properties about Bessel functions $I_n$  which can be found in \cite[Chapter 5]{Leb} and \cite{W}, and we will use along the paper. They satisfy that $I_0(0)=1$, $I_n(0)=0$ for $n\not =0,$ and $I_n(t)\geq 0$ for  $n\in\Z$ and $t\geq0$. Also, the function $I_n$ has the semigroup property (also called Neumann's identity) for the convolution on $\Z,$ that is, \begin{equation*}\label{Semigroup}I_n(t+s)=\sum_{m\in\Z}I_m(t)I_{n-m}(s)=\sum_{m\in\Z}I_m(t)I_{m-n}(s),\quad t,s\geq 0,\end{equation*} see \cite[Chapter II]{Feller}, and it satisfies the following differential-difference equation
\begin{equation}\label{prop3}\frac{\partial}{\partial t}I_n(t)=\frac{1}{2}\biggl(I_{n-1}(t)+I_{n+1}(t)\biggr), \quad t\in \C.\end{equation} Furthermore, for each $n\in\Z$ and $N\in\N_0$ \begin{equation}\label{asymptotic}
I_n(t)=\frac{e^t}{\sqrt{2\pi t}}\biggl(\sum_{k=0}^N\frac{(-1)^k a_{n,k}}{(2t)^k}+O\biggl(\frac{1}{t^{N+1}}\biggr)\biggr), \:\; \:|\arg t|<\pi/2,
\end{equation}
with $a_{n,0}=1$ and for $k\geq 1$ $a_{n,k}=\frac{(4n^2-1)(4n^2-3)\cdots(4n^2-(2k-1)^2)}{k!2^{2k}},$ see \cite[(5.11.10)]{Leb}. The previous big ''o'' function satisfies $\left|O\biggl(\frac{1}{t^{N+1}}\biggr)\right|\leq\frac{C_{n,N}}{t^{N+1}},$ being $C_{n,N}$ a positive constant depending on $n,N.$
In particular, see \cite{Leb}, we have that
\begin{equation}\label{asymptotic2}
I_n(t)=C \frac{e^t}{t^{1/2}}+R_n(t),
\end{equation}
where $|R_n(t)|\le  C_ne^t t^{-3/2}, $ for $t\to\infty.$

The generating function of the Bessel function $I_n$ is given by \begin{equation}\label{generating}e^{\frac{t(x+x^{-1})}{2}}=\sum_{n\in\Z}x^n I_n(t), \qquad x\neq 0, \,t\in \C.\end{equation} 
Observe that if we differentiate once, twice and four times with respect to the variable $x$ on the generating function \eqref{generating} and evaluating at $x=1$ we get \begin{equation}\label{der1}0=\sum_{n\in\Z}n I_{n}(t), \end{equation} \begin{equation}\label{der2}e^{t}t=\sum_{n\in\Z}n(n-1)I_{n}(t),\end{equation}
\begin{equation*}\label{der4}e^{t}(3t^2+12 t)=\sum_{n\in\Z}(n+2)(n+1)n(n-1)I_{n+2}(t).
\end{equation*}
Moreover, from \eqref{der1} and \eqref{der2} it follows \begin{equation*}\label{der1-2}
e^{t}t=\sum_{n\in\Z}n^2 I_{n}(t).
\end{equation*}
In general, it was proved in \cite[Theorem 3.3]{AG-CM} that, for every $k\in\N_0$,
\begin{equation}\label{polyI}
\sum_{n\in\Z}n^{2k}I_n(t)=e^t p_k(t), \qquad \sum_{n\in\Z}n^{2k+1}I_n(t)=0, \quad t>0,
\end{equation}
where each $p_k(t)$ is a polynomial of degree k with positive coefficients, $p_0(t)=1,$ and $p_k(0)=0$ for all $k\in\N.$

The following identities will be useful to define fractional powers of the discrete Laplacian,  \begin{equation}\label{IntFractBessel}\int_{0}^{\infty}\frac{e^{-ct}I_n(ct)}{t^{\gamma+1}}\,dt=\frac{(2c)^{\gamma}}{\sqrt{\pi}}\frac{\Gamma(1/2+\gamma)\Gamma(n-\gamma)}{\Gamma(n+1+\gamma)},\quad c>0,\, -1/2<\gamma<n,\end{equation} see \cite[Section 2.15.3, formula 3, p.305]{PBM2}, and \begin{equation}\label{fourier}
I_n(t)=\frac{e^t}{2\pi}\int_{-\pi}^{\pi}e^{-i n\theta}e^{-2t\sin^2\theta/2}\,d\theta,    
\end{equation}
see \cite[Proof of Proposition 1]{CGRTV}.
\vspace{0.5cm}
\subsubsection{A new important property of Bessel functions}\textcolor{white}{}\vspace{0.3cm}\newline
The Bessel function $I_n$ has the following useful integral representation \begin{equation}\label{prop4}I_n(t)=\frac{t^n}{\sqrt{\pi}2^n\Gamma(n+1/2)}\int_{-1}^1 e^{-ts}(1-s^2)^{n-1/2}\,ds,\quad n\in\N_0,\, t\geq 0,\end{equation} that we generalize in the following lemma.

\begin{lemma}\label{Lemma3.1} Let $n\in\N_0.$ Then for all $j\in\N$ such that $n-j\in\N_0$ one can write \begin{equation}\label{Iterada}I_n(t)=\frac{(-1)^j t^{n-j}}{\sqrt{\pi}2^{n-j}\Gamma(n+1/2-j)}\int_{-1}^1 e^{-ts}s\frac{Q_{j-1}(s,t)}{t^{j-1}}(1-s^2)^{n-1/2-j}\,ds,\end{equation} with $Q_{j-1}(s,t)=\sum_{k=0}^{j-1}c_{j-1-k,j-1}(st)^k.$ \end{lemma}
\begin{proof}
By \eqref{prop4} it follows easily integrating by parts that $\eqref{Iterada}$ holds for $j=1$ and $Q_0(s,t)=1$, where we have differentiated $(1-s^2)^{n-1/2}$ and integrated $e^{-st}.$ Doing the same procedure, differentiating  $(1-s^2)^{n-1/2-j}$, integrating $e^{-st}sQ_{j-1}(s,t)$ and denoting
\begin{equation}\label{Pj}
Q_j(s,t):=-t^2e^{ts}\biggl( \int e^{-wt}w Q_{j-1}(w,t)\,dw\biggr)_{\Big|_s}
\end{equation}
 one gets $Q_1(s,t)=st+1,Q_2(s,t)=s^2t^2+3st+3,$ which satisfy \eqref{Iterada} for $j=2,3.$ 

Thus by iterating the previous arguments we get, for $j\geq 3$, $n-(j+1)\in\N_0$, that
\begin{small}
\begin{align*}
I_n(t)&=\frac{(-1)^j t^{n-(j+1)}}{\sqrt{\pi}2^{n-(j+1)}\Gamma(n-1/2-j)}\int_{-1}^1 \frac{s}{t^{j-2}}\biggl( \int e^{-wt}w Q_{j-1}(w,t)\,dw\biggr)_{\Big|_s} (1-s^2)^{n-1/2-(j+1)}\,ds\\ \\
&=\frac{(-1)^{j+1} t^{n-(j+1)}}{\sqrt{\pi}2^{n-(j+1)}\Gamma(n-1/2-j)}\int_{-1}^1 \frac{e^{-ts}s}{t^{j}}(-t^2e^{ts})\biggl( \int e^{-wt}w Q_{j-1}(w,t)\,dw\biggr)_{\Big|_s} (1-s^2)^{n-1/2-(j+1)}\,ds\\
&=\frac{(-1)^{j+1} t^{n-(j+1)}}{\sqrt{\pi}2^{n-(j+1)}\Gamma(n-1/2-j)}\int_{-1}^1 \frac{e^{-ts}s}{t^{j}}Q_j(s,t) (1-s^2)^{n-1/2-(j+1)}\,ds.
\end{align*}
\end{small}
Moreover, if for some $j\in\N$ we can write  $Q_{j-1}(s,t):=\sum_{k=0}^{j-1}c_{j-1-k,j-1}(st)^k$ for certain coefficients $c_{j-1-k,j-1}\ge 0$, then by \cite[Section 1.3.2, formula 6, p.137]{PBM}  it follows that 
\begin{align}\label{P}
\nonumber Q_j(s,t)&=-e^{st}\sum_{k=1}^jc_{j-k,j-1}t^{k+1} \biggl(\int e^{-wt}w^{k} \,dw\biggr)_{\Big|_s}\\
\nonumber &=\sum_{k=1}^jc_{j-k,j-1}k!\sum_{m=0}^k \frac{(st)^{k-m}}{(k-m)!}=\sum_{k=1}^jc_{j-k,j-1}k!\sum_{m=0}^k \frac{(st)^{m}}{m!}\\
&=\sum_{k=1}^{j}c_{j-k,j-1} k!+\sum_{m=1}^j \frac{(st)^{m}}{m!}\sum_{k=m}^{j}c_{j-k,j-1} k!=\sum_{m=0}^j\left( \sum_{k=\max\{1,m\}}^j \frac{k!}{m!}c_{j-k,j-1}\right) (sz)^m\nonumber\\
&:=\sum_{m=0}^jc_{j-m,j} (sz)^m,
\end{align}
and the proof is over.
\end{proof}

\begin{remark}\label{Remark1} 
Note that, by \eqref{P}, if $Q_{j}(s,t)=\sum_{k=0}^{j}c_{j-k,j}(st)^k,$ being \newline $\displaystyle c_{j-k,j}=\sum_{m=\max\{1,k\}}^j \frac{m!}{k!}c_{j-m,j-1},  $ it follows that $c_{0,j}=\frac{j!}{j!}c_{0,j-1}=\frac{(j-1)!}{(j-1)!}c_{0,j-2}=\ldots=c_{0,0}=1,$ and also $c_{j,j}=c_{j-1,j},$ for all $j\in\N.$

Also, note that since \eqref{Pj} holds, then $$\frac{tQ_j(s,t)-\frac{d}{ds}Q_j(s,t)}{t^2}=sQ_{j-1}(s,t),$$ and therefore a few calculations give \begin{equation}\label{recu}
c_{k,j}=c_{k,j-1}+c_{k-1,j}(j-k+1),\quad k=1,\ldots,j-1.
\end{equation}

Note that by the recurrence formula \eqref{recu}, one gets  $$c_{1,j}=c_{1,j-1}+jc_{0,j}=c_{1,j-1}+j=c_{1,j-2}+(j-1)+j=\ldots=c_{0,1}+2+\ldots+j=\frac{j(j+1)}{2},$$ and \begin{eqnarray*}
c_{2,j}&=&c_{2,j-1}+(j-1)c_{1,j}=c_{2,j-1}+\frac{(j-1)j(j+1)}{2}=\ldots\\
&=&c_{2,2}+\frac{2\cdot3\cdot4}{2}+\ldots+\frac{(j-1)j(j+1)}{2}=\frac{1}{2\cdot4}(j-1)j(j+1)(j+2),
\end{eqnarray*} where we have applied $c_{2,2}=c_{1,2}=\frac{2\cdot3}{2}.$ In general, it follows by induction that 
\begin{equation}\label{general}  
c_{k,j}=\frac{1}{\prod_{v=1}^{k}(2v)}(j-k+1)\cdots(j+k),\quad k=1,\ldots, j-1.
\end{equation}

\end{remark}

\subsection{Discrete Heat kernel}

As we have said, $G(t,n)=e^{-2t}I_n(2t)$ is the fundamental solution of the heat problem on $\Z$, \eqref{eq1} (it is a straightforward consequence of \eqref{prop3}). In the following we present some key properties for this heat kernel.

From the theory of Confluent Hypergeometric Functions, see \cite[Section 9.11]{Leb} we have \begin{equation}\label{Hyper}
\int_0^1 e^{-4ts} s^{\gamma-\alpha-1}(1-s)^{\alpha-1}\,ds=\Gamma(\gamma-\alpha)e^{-4t}\sum_{k=0}^{\infty}\frac{(4t)^k}{k!}\frac{\Gamma(\alpha+k)}{\Gamma(\gamma+k)},\quad \Real \gamma>\Real \alpha>0
\end{equation} 
and therefore by \eqref{prop4} and a change of variable, one gets,  for $n\in\N_0,$ and  $t\geq 0$, 
\begin{align}\label{Bess}
G(t,n)&=\frac{t^n 4^n}{\sqrt{\pi}\Gamma(n+1/2)}\int_0^1 e^{-4ts}s^{n-1/2}(1-s)^{n-1/2}\,ds\nonumber\\&=\frac{1}{\sqrt{4\pi t}\Gamma(n+1/2)}\int_0^{4t} e^{-u}u^{n-1/2}\left(1-\frac{u}{4t}\right)^{n-1/2}\,du 
=\frac{t^n4^n}{\sqrt{\pi}}e^{-4t}\sum_{k=0}^{\infty}\frac{(4t)^k}{k!}\frac{\Gamma(n+1/2+k)}{\Gamma(2n+1+k)}.
\end{align}
\\
In the following two lemmata we prove new pointwise estimates for the difference of any order of $G(t,n).$

\begin{lemma}\label{pointbessel}Let $l\in\N_0,$ and $n\in\Z,$ then $$|\delta_{right}^lG(t,n)|\leq \frac{C_n}{t^{[(l+1)/2]+1/2}},\quad t>0.$$
\end{lemma}
\begin{proof}
Note that for $l=0$ the result follows by \eqref{asymptotic} taking $N=0$. If $l\in\N,$ take $N=[(l+1)/2],$  then \begin{eqnarray*}\delta_{right}^lG(n,t)&=&\frac{1}{2\sqrt{\pi t}}\sum_{j=0}^l\binom{l}{j}(-1)^j\biggl(\sum_{k=0}^{N}\frac{(-1)^k a_{n+j,k}}{(4t)^k}+O\biggl(\frac{1}{t^{N+1}}\biggr)\biggr)\\
&=&\frac{1}{2\sqrt{\pi t}}\sum_{k=0}^{N}\frac{(-1)^k}{(4t)^k}\sum_{j=0}^l\binom{l}{j}(-1)^ja_{n+j,k}+O\biggl(\frac{1}{t^{N+3/2}}\biggr).
\end{eqnarray*}
Note that for $k=0,\ldots,N$, 
$a_{n+j,k}$ is a polynomial in $j$ of order $2k,$ so we can write $a_{n+j,k}=\sum_{p=0}^{2k}\gamma_{p,n,k}\;j(j-1)\cdots(j-p+1),$ being $\gamma_{p,n,k}$ real coefficients (for $p=0$ we have the constant term $\gamma_{0,n,k}$). Then 
\begin{eqnarray*}\sum_{j=0}^l\binom{l}{j}(-1)^j a_{n+j,k}&=& \sum_{j=0}^l\binom{l}{j}(-1)^j\sum_{p=0}^{\min\{2k,j\}}\gamma_{p,n,k}j(j-1)\cdots(j-p+1) \\
&=&\sum_{p=0}^{\min\{2k,l\}}\gamma_{p,n,k}\sum_{j=p}^{l}\binom{l}{j}(-1)^j j(j-1)\cdots(j-p+1)\\
&=&\sum_{p=0}^{\min\{2k,l\}}\beta_{p,n,k,l}\sum_{j=0}^{l-p}\binom{l-p}{j}(-1)^j,
\end{eqnarray*}
with $\beta_{p,n,k,l}$ real coefficients. Since $k=0,\ldots, N,$ with $N=[(l+1)/2],$ it is not difficult to see that previous expression is null whenever $k=0,\ldots, N-1,$ and it is not null when $k=N.$ Therefore  $$\delta_{right}^lG(n,t)=\frac{C_n}{t^{1/2+[(l+1)/2]}}+O\biggl(\frac{1}{t^{3/2+[(l+1)/2]}}\biggr),$$ and the result follows.
\end{proof}

\begin{lemma}\label{heatbounds} Let $l\in\N,$ and $n\in\N_0.$ Then $$|\delta_{right}^lG(t,n)|\leq\frac{C_{l}}{t^{l/2}}\sum_{u=0}^{[l/2]}\biggl(\frac{(n+1/2)^{2}}{t}\biggr)^{l/2-u}G(t,n+l-2u)+C_{l}G(t,n)\sum_{u=[l/2]+1}^{l-1}\frac{1}{t^u},$$ being $C_l$ a positive constant which is independent on $t$ and $n.$
\end{lemma}
\begin{proof}
Let $n\in\N_0,\,t\geq 0.$ Note that by Lemma \ref{Lemma3.1} we have
\begin{eqnarray*}
\delta_{right}^{l}I_n(t)&=&\sum_{j=0}^{l}\binom{l}{j}(-1)^j I_{n+j}(t)\\
&=&\frac{t^{n}}{\sqrt{\pi}2^n\Gamma(n+1/2)}\int_{-1}^1 e^{-ts}(1-s^2)^{n-1/2}\biggl(1+\sum_{j=1}^l\binom{l}{j}\frac{sQ_{j-1}(s,t)}{t^{j-1}}\biggr)\,ds.
\end{eqnarray*}
Now, by Remark \ref{Remark1}, we write \begin{eqnarray*}
1+\sum_{j=1}^l\binom{l}{j}\frac{sQ_{j-1}(s,t)}{t^{j-1}}&=&1+\sum_{j=1}^l\binom{l}{j}\sum_{k=0}^{j-1}c_{j-k-1,j-1}s^{k+1}t^{k+1-j}\\
&=&1+\sum_{j=1}^l\binom{l}{j}\sum_{u=0}^{j-1}c_{u,j-1}\frac{s^{j-u}}{t^u}\\
&=&1+\sum_{j=1}^l\binom{l}{j}s^j+\sum_{u=1}^{l-1}\frac{1}{t^u}\sum_{j=u+1}^l\binom{l}{j}c_{u,j-1}s^{j-u}\\
&=&(s+1)^{l}+\sum_{u=1}^{l-1}\frac{s}{t^u}\sum_{k=0}^{l-u-1}\binom{l}{u+k+1}c_{u,k+u}s^{k}.
\end{eqnarray*}
Observe that $c_{1,k+1}=\frac{(k+1)(k+2)}{2}$ and therefore $$\displaystyle\sum_{k=0}^{l-2}\binom{l}{k+2}c_{1,k+1}s^{k}=l(l-1)\sum_{k=0}^{l-2}\binom{l-2}{k}s^{k}=l(l-1)(1+s)^{l-2}.$$
Now consider the case  $u=2,\ldots,l-1$. Taking into account \eqref{general} we have 
\begin{small}
$$\sum_{k=0}^{l-u-1}\binom{l}{u+k+1}c_{u,k+u}s^{k}=\frac{1}{\prod_{v=1}^{u}(2v)}l(l-1)\cdots(l-u)\sum_{k=0}^{l-u-1}\binom{l-u-1}{k}(k+u+2)\cdots(k+2u)s^k.
$$
\end{small}
Since  $(k+u+2)\cdots(k+2u)$ is a polynomial in $k$ of order $u-1,$ we can write $(k+u+2)\cdots(k+2u)=\sum_{p=0}^{u-1}b_{p,u}k(k-1)\ldots(k-p+1),$ being $b_{p,u}$ real coefficients (for $p=0,$ we have the constant term $b_{0,u}$). Then,
\begin{small}\begin{displaymath}\begin{array}{l} 
\displaystyle \sum_{k=0}^{l-u-1}\binom{l-u-1}{k}(k+u+2)\cdots(k+2u)s^k=\sum_{k=0}^{l-u-1}\binom{l-u-1}{k}\sum_{p=0}^{\min\{u-1,k\}}b_{p,u}k(k-1)\ldots(k-p+1)s^k\\
\displaystyle=\sum_{p=0}^{\min\{u-1,l-u-1\}}b_{p,u}\sum_{k=p}^{l-u-1}\binom{l-u-1}{k}k(k-1)\ldots(k-p+1)s^k\\
\displaystyle=\sum_{p=0}^{\min\{u-1,l-u-1\}}b_{p,u}(l-u-1)\cdots(l-u-p)\sum_{k=p}^{l-u-1}\binom{l-u-1-p}{k-p}s^k=\sum_{p=1}^{\min\{u,l-u\}}d_{p,u,l}s^{p-1}(s+1)^{l-u-p},
\end{array}\end{displaymath}
\end{small}
with $d_{p,u,l}\in\R.$ Therefore, we have that
\begin{align*}&\delta_{right}^{l}I_n(t)\\&=\frac{t^{n}}{\sqrt{\pi}2^n\Gamma(n+1/2)}\int_{-1}^1 e^{-ts}(1-s^2)^{n-1/2}\biggl((s+1)^l+\sum_{u=1}^{l-1}\frac{1}{t^u}\sum_{p=1}^{\min\{u,l-u\}}d_{p,u,l}s^{p}(s+1)^{l-u-p}\biggr)\,ds.
\end{align*}

Taking into account that $|s|\leq 1$ for $s\in[-1,1],$ by a change of variable we have 
\begin{align*}
|\delta_{right}^{l}G(t,n)|&\leq C_{l}\frac{t^{n}4^n}{\sqrt{\pi}\Gamma(n+1/2)}\int_{0}^1 e^{-4ts}s^{n-1/2}(1-s)^{n-1/2}\biggl(s^l+\sum_{u=1}^{l-1}\frac{1}{t^u}\sum_{p=1}^{\min\{u,l-u\}}s^{l-u-p}\biggr)\,ds\\
&\leq C_{l}\sum_{u=0}^{[l/2]}\frac{t^{n-u}4^n}{\sqrt{\pi}\Gamma(n+1/2)}\int_{0}^1 e^{-4ts}s^{n-1/2+l-2u}(1-s)^{n-1/2}\,ds\\
&+C_{l}\sum_{u=[l/2]+1}^{l-1}\frac{t^{n-u}4^n}{\sqrt{\pi}\Gamma(n+1/2)}\int_{0}^1 e^{-4ts}s^{n-1/2}(1-s)^{n-1/2}\,ds\\
&\leq C_{l}\sum_{u=0}^{[l/2]}\frac{t^{n-u}4^n\Gamma(n+l-2u+1/2)}{\sqrt{\pi}\Gamma(n+1/2)} e^{-4t}\sum_{k=0}^{\infty}\frac{(4t)^k\Gamma(n+1/2+k)}{k!\Gamma(2n+l+1-2u+k)}\\
&\quad+C_{l}G(t,n)\sum_{u=[l/2]+1}^{l-1}\frac{1}{t^u},
\end{align*}
where in the last inequality we have used \eqref{Hyper}. Now take $u=0,\ldots,[l/2],$ and note that $m:=l-2u$ is positive. An easy computation shows that $$\frac{\Gamma(n+1/2+k)}{\Gamma(2n+l+1-2u+k)}\leq C_l\frac{\Gamma(n+m+1/2+k)}{\Gamma(2n+2m+1+k)},$$ for all $n,k\in\N_0.$ Also, $\frac{\Gamma(n+l-2u+1/2)}{\Gamma(n+1/2)}\leq C_{l}(n+1/2)^{m}.$ Then by \eqref{Bess} we have \begin{eqnarray*}
|\delta_{right}^{l}G(t,n)|&\leq&C_{l}\sum_{u=0}^{[l/2]}\frac{t^{n-u}4^n(n+1/2)^{l-2u}}{\sqrt{\pi}} e^{-4t}\sum_{k=0}^{\infty}\frac{(4t)^k\Gamma(n+l-2u+1/2+k)}{k!\Gamma(2(n+l-2u)+1+k)}\\
&&+C_{l}\sum_{u=[l/2]+1}^{l-1}\frac{G(t,n)}{t^u}\\
&\leq&\frac{C_{l}}{t^{l/2}}\sum_{u=0}^{[l/2]}\biggl(\frac{(n+1/2)^{2}}{t}\biggr)^{l/2-u}G(t,n+l-2u)+C_{l}G(t,n)\sum_{u=[l/2]+1}^{l-1}\frac{1}{t^u}.
\end{eqnarray*}

\end{proof}

\begin{remark}\label{remheatbounds}
Let $l,n\in\N$, and $t>1$. From the previous lemma,  if $1\leq n\leq \sqrt{t}-1/2$, then $$|\delta_{right}^lG(t,n)|\leq C\frac{ G(t,n)}{t^{l/2}},$$ and if $n\geq \sqrt{t}-1/2$,  then $$|\delta_{right}^lG(t,n)|\leq C\frac{ G(t,n)(n+1/2)^{l}}{t^{l}}.$$ 
In particular, the previous bounds are also valid when $t\in(0,1]$ since $G(t,j)$ is decreasing for $j\in\N_0,$ therefore $|\delta_{right}^lG(t,n)|\leq CG(t,n),$ and $G(t,n)\leq \frac{G(t,n)}{t^{l/2}}$ and $G(t,n)\leq C\frac{G(t,n)(n+1/2)^l}{t^{l}}$ for all $t\in(0,1]$ and $n\in\N.$


\end{remark}

The following result shows decay rates for the $\ell^1$-norm of the differences of any order. The first difference was proved in \cite[Theorem 4.3]{AG-CM}. One should keep in mind that $\sum_{j\in\Z}G(t,j)=1$. 

\begin{lemma}\label{kernelest} 
Let $l\in\N$ and $t>0,$ then 
$$\sum_{j\in\Z}|\delta_{right}^l G(t,j)|:=\|\delta_{right}^l G (t,\cdot)\|_1\leq C\min\{1,\frac{1}{t^{l/2}}\}.$$
\end{lemma}
\begin{proof}
Let $t >0.$ If $t\in(0,1],$ then it is clear that $\sum_{j\in\Z}|\delta_{right}^lG(t,j)|\leq C.$ Now let $t>1.$ Then $$
\sum_{j\in\Z}|\delta_{right}^l G(t,j)|=\biggl(\sum_{|j|\leq l}+\sum_{|j|>l}\biggr)|\delta_{right}^l G(t,j)|:=I+II.
$$ On the one hand, by Lemma \ref{pointbessel} we have that
$$I\leq \frac{C}{t^{1/2+[(l+1)/2]}}\leq \frac{C}{t^{1/2+l/2}}\leq \frac{C}{t^{l/2}}.$$
On the other hand, since $|\delta_{right}^l G(t,j)|=|\delta_{right}^l G(t,|j|-l)|$ for $j\le -l,$ we have that 
\begin{align*}II &=\sum_{|j|>l}|\delta_{right}^l G(t,j)|\leq 2 \sum_{j\geq 1}|\delta_{right}^l G(t,j)|\leq 2 \biggl(\sum_{1\leq j\leq \sqrt{t}-1/2}+\sum_{\substack{j>\sqrt{t}-1/2\\ j\ge 1}}\biggr)|\delta_{right}^l G(t,j)|\\&:=II.1+II.2.
\end{align*}
Let now $k$ the least natural number such that $2k\geq l.$ Then by Lemma \ref{heatbounds} and Remark \ref{remheatbounds} we have $$II.1\leq \frac{C}{t^{l/2}}\sum_{1\leq j\leq  \sqrt{t}-1/2}G(t,j)\leq \frac{C}{t^{l/2}},$$ and 
$$
II.2\leq \frac{C}{t^{l}}\sum_{\substack{j>\sqrt{t}-1/2\\ j\ge 1}}G(t,j)(|j|+1/2)^l\\
\leq \frac{C}{t^{l/2+k}}\sum_{\substack{j>\sqrt{t}-1/2\\ j\ge 1}}G(t,j)j^{2k}\leq \frac{C p_k(2t)}{t^{l/2+k}}\leq \frac{C}{t^{l/2}}.
$$
\end{proof}

\begin{remark}\label{Remkernelest} Note that the $\ell^1$-norm of the even differences $2l$ of $G(t,n)$ can be seen as the $\ell^1$-norm of the derivative or order $l$  in time. In this case our result coincides with the ones proved in \cite{I2} and \cite{I}. 
\end{remark}

\subsection{Discrete Poisson kernel}

The discrete Poisson kernel is defined as
$$
P(y,j):=
\frac{y}{2\sqrt{\pi}}\int_0^\infty \frac{e^{-\frac{y^2}{4t}}}{t^{3/2}} G(t,j)\,dt,\quad j\in\Z, \:\;y>0.
$$
Since we don't have an explicit expression in terms of known functions for $P(y,j),$ we will take advantage of the subordination formula and the properties we have proved for the heat kernel,  $G(t,j),$ to get some estimates of the Poisson kernel and, more generally, of subordinate functions of type $y\frac{e^{-\frac{cy^2}{t}}}{t^{3/2}},$ being $c>0$ an arbitrary constant. Along the paper, from one inequality to other, the constant $c$ could be different, but we will not change it on them for convenience.

Observe that, for every $l\in\N,$ $y,t>0$
\begin{equation}\label{Poissontrick}\displaystyle\Bigg|\partial_y^l\Bigg( \frac{{y}e^{-\frac{cy^2}{t}}}{t^{3/2}}\Bigg)\Bigg|\le C \frac{e^{-\frac{cy^2}{t}}}{t^{l/2+1}}\le \frac{C}{y^{l-1}}\; \frac{y^{l-1}e^{-\frac{cy^2}{t}}}{t^{\frac{l-1}{2}}t^{3/2}}\le \frac{Cy}{y^{l}}\frac{e^{-\frac{cy^2}{t}}}{t^{3/2}}, \end{equation} and for every $c>0,$ $\displaystyle\int_0^{\infty}y\frac{e^{-\frac{cy^2}{t}}}{t^{3/2}}\,dt=C<\infty.$

Next results state pointwise estimates for Poisson-type discrete kernels.

\begin{lemma}\label{Poissonlema}
Let $y,c>0$ and $l\in\N_0.$ The following estimates hold:\begin{itemize}

    \item[(i)] $\left|\partial_y^l\biggl(\int_0^{\infty}y\frac{e^{-\frac{cy^2}{t}}}{t^{3/2}}G(t,j)\,dt\biggr)\right| \leq \frac{C}{y^{l}(1+|j|)},\quad j\in\Z.$
    \item[(ii)] $\left|\partial_y^l\biggl(\int_0^{\infty}y\frac{e^{-\frac{cy^2}{t}}}{t^{3/2}}G(t,j)\,dt\biggr)\right|\leq \frac{Cy}{y^{l}|j|^2},\quad j\neq 0.$ 
    \item [(iii)]$\left|\partial_y^l\delta_{right/left}\biggl(\int_0^{\infty}y\frac{e^{-\frac{cy^2}{t}}}{t^{3/2}}G(t,j)\,dt\biggr)\right| \leq \frac{C}{y^{l+2}},\quad j\in\Z.$
    \item [(iv)]$\left|\partial_y^l\delta_{right/left}\biggl(\int_0^{\infty}y\frac{e^{-\frac{cy^2}{t}}}{t^{3/2}}G(t,j)\,dt\biggr)\right| \leq \frac{C}{y^{l}|j|^2},\quad j\neq 0$

\end{itemize} 
\end{lemma}
\begin{proof}
First we prove epigraphs (i) and (ii). Observe that, from \cite[Lemma 4.1 (i)]{AG-CM} we have that, for  $0<t\le |j|^2,$ $|G(t,j)|\le C\frac{t}{|j|^3},$ and for $t>0$, $G(t,j)\le G(t,0)\le \frac{C}{t^{1/2}},$ for $j\in\Z.$  In virtue of \eqref{Poissontrick}, it follows that, for $j\neq 0,$ \begin{align*}
    &\left|\partial_y^l\biggl(\int_0^{\infty}y\frac{e^{-\frac{cy^2}{t}}}{t^{3/2}}G(t,j)\,dt\biggr)\right|\le \frac{C}{y^l}\biggl(\int_0^{|j|^2} y\frac{e^{-\frac{cy^2}{t}}}{t^{3/2}}\frac{t}{|j|^3}dt+\int_{|j|^2}^\infty y\frac{e^{-\frac{cy^2}{t}}}{t^{3/2}}\frac{1}{t^{1/2}}dt\biggr)\\
    \nonumber&\le C \frac{C}{y^l|j|}\int_0^{\infty} y\frac{e^{-\frac{cy^2}{t}}}{t^{3/2}}dt\leq \frac{C}{y^l(1+|j|)},
    \end{align*}
and  \begin{align*}
    &\left|\partial_y^l\biggl(\int_0^{\infty}y\frac{e^{-\frac{cy^2}{t}}}{t^{3/2}}G(t,j)\,dt\biggr)\right|\le \frac{Cy}{y^l}\biggl(\int_0^{|j|^2} \frac{1}{t^{3/2}}\frac{t}{|j|^3}dt+\int_{|j|^2}^\infty \frac{1}{t^{2}}dt\biggr)\le C \frac{y}{y^l|j|^2}.
    \end{align*}
The case $j=0$ in part (i) follows from \eqref{Poissontrick} and the fact that $G(t,0)\leq C$ for $t>0.$

Now we prove epigraphs (iii) and (iv). Since $G(t,j)=G(t-j)$ for $j\in\N$, and $G(t,j+1)\leq G(t,j)$ for $j\in\N_0$,  from Lemma \ref{heatbounds}  we get that
\begin{align*}
|\delta_{right}G(t,j)|&\leq C\frac{j+1/2}{t}G(t,j+1)\leq C\frac{j+1/2}{t}G(t,j), \:\; \text{ for } j\in\N_0\\\text{ and } \\
|\delta_{right}G(t,j)|&=|G(t,|j|)-G(t,|j|-1)|
\leq C\frac{|j|+1/2}{t}G(t,|j|), \text{ for } j\leq -1.
\end{align*}
Therefore, we have that \begin{equation*}\label{firstdifference}
  |\delta_{right}G(t,j)|\leq C\frac{(|j|+1/2)}{t}G(t,|j|),\quad \text{for every } j\in\Z \text{ and }\:t>0.
  \end{equation*}
Also, we have for $t>0$ and $j\neq 0,$
$$|\delta_{right}G(t,j)|\le C\left\{ \begin{array}{c} 
\frac{|j|}{t^{3/2}},  \text{ if } |j|^2\le  t,\\
\frac{t}{|j|^{4}}, \text{ if } |j|^2\ge  t,
\end{array}\right.
$$ see \cite[Lemma 4.1 (ii)]{AG-CM}. Thus, by  using \eqref{Poissontrick} and the estimates above we get that, for $j\in\Z,$ \begin{align*}
    &\left|\partial_y^l\delta_{right}\biggl(\int_0^{\infty}y\frac{e^{-\frac{cy^2}{t}}}{t^{3/2}}G(t,j)\,dt\biggr)\right|\le \frac{C}{y^l}\int_0^{\infty}y\frac{e^{-\frac{cy^2}{t}}}{t^{3/2}}\frac{|j|+1/2}{t}G(t,j)\,dt \\
    &\leq \frac{C(|j|+1/2)}{y^{l+2}}\int_0^{\infty}y\frac{e^{-\frac{cy^2}{t}}}{t^{3/2}}\frac{y^2}{t}G(t,j)\,dt \leq \frac{C(|j|+1/2)}{y^{l+2}}\int_0^{\infty}y\frac{e^{-\frac{cy^2}{t}}}{t^{3/2}}G(t,j)\,dt\\
    &\leq \frac{C}{y^{l+2}},
    \end{align*}
where in the last inequality we have proceeded as in epigraph (i), and for $j\neq 0$ \begin{align*}
    &\left|\partial_y^l\delta_{right}\biggl(\int_0^{\infty}y\frac{e^{-\frac{cy^2}{t}}}{t^{3/2}}G(t,j)\,dt\biggr)\right|\le \frac{C}{y^l}\biggl(\int_0^{|j|^2} y\frac{e^{-\frac{cy^2}{t}}}{t^{3/2}}\frac{t}{|j|^4}dt+\int_{|j|^2}^\infty y\frac{e^{-\frac{cy^2}{t}}}{t^{3/2}}\frac{|j|}{t^{3/2}}dt\biggr)\\
    \nonumber&\le  \frac{C}{y^l|j|^2}\int_0^{\infty} y\frac{e^{-\frac{cy^2}{t}}}{t^{3/2}}dt= \frac{C}{y^l|j|^2}.
    \end{align*}

The estimates with $\delta_{left}$ in (iii) and (iv) are analogous.
\end{proof}


From Lemma \ref{kernelest} we can get the $\ell^1$-norm estimates for the Poisson semigroup.

\begin{lemma}\label{Poissonest} Let $y>0.$ \\

\begin{itemize}
\item[(1)]$\displaystyle \sum_{j\in\Z}|\partial_y^mP(y,j)|:=\|\partial_y^mP (y,\cdot)\|_1\leq Cy^{-m},$ for every $m\in\N_0.$\\

\item[(2)]$\displaystyle \sum_{j\in\Z}|\delta_{right}^lP(y,j)|:=\| \delta_{right}^l P(y,\cdot)\|_1\leq Cy^{-l},$ for every $l\in\N_0.$
\end{itemize}
\end{lemma}

\begin{proof} 

Let $y>0$ and $m\in\N_0.$ By \eqref{Poissontrick} and Lemma \ref{kernelest} we have that
\begin{align*}
&\sum_{j\in\Z}|\partial_y^mP(y,j)|\le \frac{C}{y^m} \int_0^\infty y\frac{e^{-\frac{cy^2}{t}}}{t^{3/2}}\sum_{j\in\Z}G(t,j)dt =\frac{C}{y^m}.\\
\end{align*}
On the other hand, Lemma \ref{kernelest} implies that
\begin{align*}
&\sum_{j\in\Z}|\delta_{right}^lP(y,j)|\le   \frac{y}{2\sqrt{\pi}}\int_0^\infty \frac{e^{-\frac{y^2}{4t}}}{t^{3/2}}\sum_{j\in\Z}|\delta_{right}^lG(t,j)|\,dt\le Cy\int_0^\infty \frac{e^{-\frac{y^2}{4t}}}{t^{3/2+l/2}}dt\le\frac{C}{y^{l}}.\\
\end{align*}

\end{proof}

\begin{remark}\label{RemPoissonest}
Taking into account that $P(y,j)$ satisfies the Poisson equation, the results in Lemma \ref{Poissonest} imply that
\begin{align*}
\sum_{j\in\Z}|\delta_{right/left}\delta_{right/left}P(y,j)|:=\| \delta_{right/left}\delta_{right/left} P(y,\cdot)\|_1=\|\partial_y^2P(y,\cdot)\|_1\leq Cy^{-2},\quad {y>0}.
\end{align*}
\end{remark}

\subsection{Heat and Poisson semigroups}\textcolor{white}{}
\newline
The following remark contains crucial observations to get our results.

\begin{remark}\label{obs}Let $f:\mathbb{Z}\to \R$ be a function such that the semigroup $e^{t\Delta_d}f$ is well defined for every $t>0$. 
Then it follows that $\delta_{right/left}e^{t\Delta_d}f$ and $\partial_t^l e^{t\Delta_d}f$ are also well defined for every $l\in\N$. In that case, we have that

\begin{align*}
    \delta_{right}e^{t\Delta_d}f(n)&=e^{t\Delta_d}f(n)-e^{t\Delta_d}f(n+1)\\
    &=\sum_{j\in\Z} G(t,n-j)f(j)-\sum_{j\in\Z} G(t,n+1-j)f(j)=\sum_{j\in\Z} (\delta_{right}G(t,n-j))f(j).
\end{align*}
Moreover, by performing a change of variables, we get that
\begin{align*}
    \delta_{right}e^{t\Delta_d}f(n)
    &=\sum_{j\in\Z} G(t,j)f(n-j)-\sum_{j\in\Z} G(t,j)f(n+1-j)=\sum_{j\in\Z} G(t,j)\delta_{right}f(n-j).
\end{align*}
The analogous identities for $\delta_{left}$ hold.

On the other hand, for $t=t_1+t_2,$ where $t,t_1,t_2>0,$ the semigroup property gives
$$
e^{t\Delta_d}f(n)=e^{t_1\Delta_d}(e^{t_2\Delta_d}f)(n)=\sum_{j\in\Z}G(t_1,j)e^{t_2\Delta_d}f(n-j).
$$
Furthermore, since 
\begin{align*}
    \partial_te^{t\Delta_d}f(n)|_{t=t_1+t_2}=\partial_{t_1}e^{(t_1+t_2)\Delta_d}f(n)=\partial_{t_2}e^{(t_1+t_2)\Delta_d}f(n),
    \end{align*}
    we obtain that 
    \begin{align*}
         \partial_te^{t\Delta_d}f(n)|_{t=t_1+t_2}=\sum_{j\in\Z}\partial_{t_1}G(t_1,j)e^{t_2\Delta_d}f(n-j)=\sum_{j\in\Z}G(t_1,j)\partial_{t_2}e^{t_2\Delta_d}f(n-j).
    \end{align*}

For the Poisson semigroup the analogous properties hold. If $e^{-y\sqrt{-\Delta_d}}f$ is well defined for every $y>0$, then it is clear that $\delta_{right}e^{-y\sqrt{-\Delta_d}}f$ is well defined. Also, if $$\sum_{j\in\Z}y\biggl(\int_{0}^{\infty}\frac{e^{-\frac{cy^2}{t}}}{t^{3/2}}G(t,j)\biggr)|f(n-j)|<\infty$$ for each $y>0,\:n\in\Z,$ and $c>0,$ then, by \eqref{Poissontrick}, $\partial_y^le^{-y\sqrt{-\Delta_d}}f$ is well defined for every $l\in\N$ and the properties above stated for the discrete heat semigroup are also fulfilled for the Poisson semigroup.
\end{remark}

Next we study functions for which the semigroups are well defined.

\begin{lemma}\label{semigroupP}
Let  $f:\Z\to\R$.
\begin{itemize}
\item[A.] Suppose that for certain $\alpha>0,$ $\frac{|f|}{1+|\cdot|^\alpha}\in\ell^\infty(\Z).$ Then,
\begin{itemize}
   \item[(i)] For every $t>0$,  $e^{t\Delta_d}f$ is well defined and $${|e^{t\Delta_d}f(n)|}\le C{(1+|n|^\alpha+t^{\alpha/2}}),\quad n\in\Z.$$ 
        \item[(ii)] For every $l\in\N,$ and $t>0,$ $$|\delta_{right}^le^{t\Delta_d}f(n)|\le C\left((1+|n|^{\alpha})\min\left\{1,\frac{1}{t^{l/2}}\right\}+t^{\alpha/2-l/2}\right),  \quad n\in\Z.$$
        \item[(iii)] $\displaystyle\lim_{t\to 0} e^{t\Delta_d}f(n)=f(n)$, for every $n\in\Z.$ 
                    \end{itemize}
                    \item[B.] Suppose that $f$ satisfies $\sum_{j\in\Z}\frac{|f(j)|}{1+|j|^2}<\infty$. Then,  $e^{-y\sqrt{-\Delta_d}}f$ is well defined for every $y>0$ and
       $\displaystyle\lim_{y\to 0} e^{-y\sqrt{-\Delta_d}}f(n)=f(n),$ for every $n\in\Z.$ 
 \end{itemize}
\end{lemma}
\begin{proof}
We start proving A.(i). Let $t>0$, $n\in\Z$ and $m$ be the smallest natural number such that $2m\geq\alpha$. By using \eqref{polyI}, we have that
 \begin{align*}
     |e^{t\Delta_d}f(n)|&\le C\sum_{j\in\Z} G(t,j)(1+|n-j|^\alpha)\le C\sum_{j\in\Z} G(t,j)(1+|n|^\alpha+|j|^\alpha)\nonumber\\
     &\le C\left(1+|n|^\alpha+\sum_{|j|\le \sqrt{t}} G(t,j)|j|^\alpha+\sum_{|j|>\sqrt{t}} G(t,j)|j|^{\alpha}\min\left\{\frac{|j|}{\sqrt{t}}, |j|\right\}^{2m-\alpha}\right)\nonumber\\
     &= C\left(1+|n|^\alpha+t^{\alpha/2}+p_{m}(2t)\min\left\{\frac{1}{t^{m-\alpha/2}},1 \right\}\right)
     \le C (1+|n|^\alpha+t^{\alpha/2}).
 \end{align*}
In the last inequality we have used that $|p_m(2t)|\le C,$ if $0<t<1$, and $|p_m(2t)|\le C t^m$ whenever $t>1.$
 
Next we prove (ii).  Let $t>0$, $n\in\Z$ and $m$ be the smallest natural number such that $2m\geq l +\alpha$. Then since $\sum_{j\in\Z}|\delta_{right}^lG(t,j)|\leq C\min\{1,\frac{1}{t^{l/2}}\}$ (see Lemma \ref{kernelest}) we have that
  \begin{align*}
     | \delta_{right}^le^{t\Delta_d}f(n)|&\le C \sum_{j\in\Z}|\delta_{right}^lG(t,j)|(1+|n|^{\alpha}+|j|^{\alpha})\\
     & \leq C(1+|n|^{\alpha})\sum_{|j|\leq l}|\delta_{right}^lG(t,j)|+C\sum_{|j|> l}|\delta_{right}^lG(t,j)|(1+|n|^{\alpha}+|j|^{\alpha})\\
     &\leq C((1+|n|^{\alpha})\min\left\{1,\frac{1}{t^{l/2}}\right\}+C\sum_{|j|>l}|\delta_{right}^lG(t,j)| |j|^{\alpha}.
  \end{align*}
Recall that if $j< -l,$ one can write $|\delta_{right}^q G(t,j)|=|\delta_{right}^q G(t,|j|-l)|,$ with $|j|-l\geq 1,$ so by Remark \ref{remheatbounds}  and Lemma \ref{kernelest}, it follows that
 \begin{align*}
     \sum_{|j|>l}|\delta_{right}^lG(t,j)| |j|^{\alpha}&\leq \sum_{j\geq 1}|\delta_{right}^lG(t,j)|(|j+l|^{\alpha}+|j|^{\alpha}) \leq C\sum_{j\geq 1}|\delta_{right}^lG(t,j)||j|^{\alpha}\\
     &\leq C\biggl(t^{\alpha/2}\sum_{1\leq j\leq\sqrt{t}}|\delta_{right}^lG(t,j)|+\frac{1}{t^l}\sum_{j>\sqrt{t}\ge 1}G(t,j)|j|^{\alpha+l}+\sum_{\substack{j>\sqrt{t}\\0<t<1}}G(t,j)|j|^{\alpha}\biggr) \\
    &\leq C\biggl(t^{\alpha/2-l/2}+{p_{m}(2t)}\min\left\{\frac{1}{t^{l+m-\alpha/2-l/2}},1 \right\}\biggr) \\
    &\leq C(1+t^{\alpha/2-l/2}).
 \end{align*}
    
In the last inequality we have used that $|p_m(2t)|\le C,$ if $0<t<1$, and $|p_m(2t)|\le C t^m$ whenever $t>1.$

Now we prove (iii). Note that $$|e^{t\Delta_d}f(n)-f(n)|=\left|\sum_{j\neq 0}G(t,j)(f(n-j)-f(n))\right|\leq C\sum_{j\neq 0}G(t,j)(1+|n|^{\alpha}+|j|^{\alpha}).$$ On one hand, 
 $$\sum_{j\neq 0}G(t,j)(1+|n|^{\alpha})=(1+|n|^{\alpha})(1-G(t,0))\to 0,\quad t\to 0^+,$$ since $G(t,0)\to 1$ as $t\to 0^+.$ On the other hand, $$\sum_{j\neq 0}G(t,j)|j|^{\alpha}\leq p_m(2t)\to0,\quad t\to 0^+,$$ being now $2m\geq \alpha.$ Then the result follows.

 Finally we prove B. Let  $|n|\le A,$ $A\in\N.$  We can write
 $$
 f=f\chi_{\{|j|\le 2A\}}+f\chi_{\{|j|>2A\}}:=f_1+f_2.
 $$
 Note that when $|j|>2A$  one gets $|j|\leq  2|n-j|$. Then by Lemma \ref{Poissonlema} (ii) we have
 $$
 |e^{-y\sqrt{-\Delta_d}}f_2(n)|\le \sum_{|j|>2A}P(y,n-j)|f(j)|\le C\sum_{|j|>2A}\frac{y}{|n-j|^2}|f(j)|\le C y\sum_{|j|>2A}\frac{|f(j)|}{|j|^2}\rightarrow 0,
 $$
 as $y\to 0$.
 
 On the other hand, $f_1\in \ell^p(\Z)$ for each $p\geq 1,$ and $e^{t\Delta_d}$ is $C_0$-semigroup on $\ell^p(\Z).$ In particular it is strongly continuous at the origin in $\ell^p(\Z),$ and therefore pointwise, that is,  
  $$
  \lim_{y\to 0} e^{-y\sqrt{-\Delta_d}}f_1(n)=f_1(n)=f(n).
  $$
 We conclude that 
   {$e^{-y\sqrt{-\Delta_d}}f$ is well defined for every $y>0$ and $ \lim_{y\to 0} e^{-y\sqrt{-\Delta_d}}f(n)=f(n),$ for every $n\in\Z.$}
\end{proof}

\begin{lemma}\label{decay}
Let $f:\mathbb{Z}\to \R$. 
\begin{itemize}
    \item[1.] If $f$ satisfies $\frac{|f|}{1+|\cdot|^\alpha}\in\ell^\infty(\Z),$ for certain $\alpha>0,$  then, for every $n\in\Z,$ $m:=m_1+m_2$, with $m_1,m_2\in\N_0$ and $l\in\N_0$, such that $\frac{m}{2}+l> \alpha/2$, we have that
$$
\partial_t^l\delta_{right/left}^{m_1,m_2} e^{t\Delta_d}f(n)\to 0 ,\quad \text{ as } t\to\infty.
$$
\item[2.] If $f$ satisfies $\sum_{j\in\Z}\frac{|f(j)|}{1+|j|^2}<\infty$,  then, for every $n\in\Z,$ $m:=m_1+m_2$, with $m_1,m_2\in\N_0$ and $l\in\N_0$, such that ${m}+l\ge 1$, we have that
$$
\partial_y^l\delta_{right/left}^{m_1,m_2} e^{-y\sqrt{-\Delta_d}}f(n)\to 0, \quad \text{ as } y\to\infty.
$$
\end{itemize}

\end{lemma}
\begin{proof}
Suppose that $f$ satisfies $\frac{|f|}{1+|\cdot|^\alpha}\in\ell^\infty(\Z),$ for certain $\alpha>0$, and let 
 $n\in\Z$ and $m_1,m_2,l\in\N_0$  such that  $\frac{m}{2}+l> \alpha/2$. It follows easily that there exist $n'\in\Z$ ($n'$ is comparable to $n$) and $q=2l+m_1+m_2\in\N$ with $q>\alpha$ such that $$|\partial_t^l\delta_{right/left}^{m_1,m_2} e^{t\Delta_d}f(n)|=|\delta_{right}^{q} e^{t\Delta_d}f(n')|.$$ Then, it follows by Lemma \ref{semigroupP} A (ii) that  $\delta_{right}^{q} e^{t\Delta_d}f(n')\to 0,\quad t\to\infty.$

Now we prove 2. Suppose that $f$ satisfies $\sum_{j\in\Z}\frac{|f(j)|}{1+|j|^2}<\infty$ and let   $n\in\Z$ and $m_1,m_2,l\in\N_0$ such that ${m}+l\ge 1$. 

Suppose first  that  $m=0,$ and $l\in\N.$ 
By \eqref{Poissontrick} we have that, for every $y>0$ and $n\in\Z,$ 
\begin{align*}
|\partial_y^l e^{-y\sqrt{-\Delta_d}}f(n)|
&\le \biggl( \sum_{|j|\leq \sqrt{y}}+\sum_{|j|\geq \sqrt{y}} \biggr)|\partial_y^lP(y,j)||f(n-j)|dt=:A1+B1.
\end{align*}
Note that by Lemma \ref{Poissonlema} (i) we have that $$A1\leq \frac{C(1+\sqrt{y})}{y^l}\sum_{|j|\leq \sqrt{y}}\frac{1}{(1+|j|)^2}|f(n-j)|\to 0,\quad y\to\infty,$$ 
and by Lemma \ref{Poissonlema} (ii),
$$B1\leq \frac{C}{y^{l-1}}\sum_{|j|\geq \sqrt{y}}\frac{1}{|j|^2}|f(n-j)|\to 0,\quad y\to\infty.$$

Secondly, since $\delta_{right}^2e^{-y\sqrt{-\Delta_d}}f(n)=\Delta_de^{-y\sqrt{-\Delta_d}}f(n+1)=\partial_y^2 e^{-y\sqrt{-\Delta_d}}f(n+1)$, the case when $m$ is even follows from the previous one. 

Finally, it remains to prove that, for $l\in\N_0$ and $n\in\Z$, $\partial_y^l\delta_{right/left} e^{-y\sqrt{-\Delta_d}}f(n)\to 0$, as $y\to\infty$.
Let $l\in\N_0,$ $n\in\Z$ and $y>0.$ We write
\begin{align*}
|\partial_y^l\delta_{right} e^{-y\sqrt{-\Delta_d}}f(n)|&\le  \left(\sum_{|j|\le \sqrt{y}} +\sum_{|j|\ge \sqrt{y}}\right) |\partial_y^l\delta_{right}P(y,j)||f(n-j)|dt\\
&=A2+B2.
\end{align*}
On the one hand, by Lemma \ref{Poissonlema} (iii)  
$$A2\le \frac{C}{y^{l+2}} \sum_{|j|\le \sqrt{y}}  |f(n-j)|\le \frac{C(1+y)}{y^{l+2}} \sum_{|j|\le \sqrt{y}}\frac{1}{1+|j|^2}  |f(n-j)|\to 0,\quad y\to\infty,$$ and by Lemma \ref{Poissonlema} (iv) $$B2\le \frac{C}{y^{l}} \sum_{|j|\ge \sqrt{y}}  \frac{|f(n-j)|}{|j|^2}\to 0,\quad y\to\infty.$$

The case including $\delta_{left}$ is analogous.

\end{proof}

\begin{lemma}\label{subirk}
Let $f:\mathbb{Z}\to \R$ be a function.
\begin{itemize}
    \item [1.]Suppose that $f$ satisfies $\frac{|f|}{1+|\cdot|^\alpha}\in\ell^\infty(\Z),$ for certain $\alpha>0.$ For every  $k,l\in \N$ such that $l> k\ge [\alpha/2]+1$ and $t>0,$ the following are equivalent
\begin{itemize}
    \item[i)] $\|\partial_t^k e^{t\Delta_d}f\|_{\infty}\leq C t^{-k+\alpha/2}$.
    \item[ii)]$\|\partial_t^l e^{t\Delta_d}f\|_{\infty}\leq C t^{-l+\alpha/2}$.
\end{itemize}
    \item [2.]Suppose that $f$ satisfies $\sum_{j\in\Z}\frac{|f(j)|}{1+|j|^2}<\infty$.  For every  $p,q\in \N$ such that $p> q\ge [\alpha]+1$, $\alpha>0$ and $y>0,$ the following are equivalent
\begin{itemize}
    \item[i)] $\|\partial_y^q e^{-y\sqrt{-\Delta_d}}f\|_{\infty}\leq C y^{-p+\alpha}$.
    \item[ii)]$\|\partial_y^p e^{-y\sqrt{-\Delta_d}}f\|_{\infty}\leq C y^{-q+\alpha}$.
\end{itemize}
\end{itemize} 
\end{lemma}
\begin{proof}
We only do the proof for the heat kernel and the case $k=[\alpha/2]+1$ and $l=k+1.$ The rest of cases are analogous.

Suppose that $f$ satisfies i). Then, by the semigroup property and Lemma \ref{kernelest}, we have that
\begin{align*}
   |\partial_t^l e^{t\Delta_d}f(n)| &=C\left| \sum_{j\in\Z} \partial_u G(u,j)|_{u=t/2}\partial_v^k e^{v\Delta_d}f(n-j)|_{v=t/2}\right|\\
   &\le  C\|\partial_v^k e^{v\Delta_d}f|_{v=t/2}\|_{\infty}\sum_{j\in\Z}|\partial_u G_u(j)|_{u=t/2}|\le Ct^{-k+\alpha/2} t^{-1}=    Ct^{-l+\alpha/2}. 
\end{align*}
Conversely, suppose that ii) holds. Since for each $n\in\Z,$ $\partial_t^k e^{t\Delta_d}f(n)\rightarrow 0$ as $t\to \infty$, see Lemma \ref{decay}, we have that
$$
|\partial_t^k e^{t\Delta_d}f(n)|=\left|\int_t^\infty \partial_u^{k+1} e^{u\Delta_d}f(n)du\right|\le C t^{-k+\alpha/2}.
$$

\end{proof}

\begin{lemma}\label{cambioyx}Let $f:\Z\to\R$.
\begin{itemize}
    \item Suppose that $f\in \Lambda^\alpha_H$, for some $\alpha>0.$ Then, for every $l\in\N_0$ and $m\in\{1,2\}$ such that $\frac{m}{2}+l\ge [\alpha/2]+1$, we have that
$$
\|\partial_t^l\delta_{right/left}^{m_1,m_2} e^{t\Delta_d}f\|_{\infty}\leq C t^{-(l+\frac{m}{2})+\frac{\alpha}{2}},\qquad m_1,m_2\in\N_0, \: m_1+m_2=m,\ t>0.
$$
\item Suppose that  $f\in \Lambda^\alpha_P$, for some $\alpha>0.$ Then, for every $l\in\N_0$ and $m\in\{1,2\}$ such that ${m}+l\ge [\alpha]+1$, we have that
$$
\|\partial_y^l\delta_{right/left}^{m_1,m_2} e^{-y\sqrt{-\Delta_d}}f\|_{\infty}\leq C y^{-(l+{m})+\alpha},\qquad m_1,m_2\in\N_0, \: m_1+m_2=m,\ y>0.
$$
\end{itemize}

\end{lemma}
\begin{proof}
We only do the proof for the heat semigroup. For the Poisson is completely analogous.
Suppose that $f\in \Lambda^\alpha_H$, for some $\alpha>0.$ We consider first the case $l\ge [\alpha/2]+1$, $m\in \{1,2\}.$ From the semigroup property,  Lemmata \ref{kernelest}, \ref{subirk} and Remark \ref{Remkernelest}, we have that
\begin{align}\label{derivcomb}
    |\partial_t^l\delta_{right/left}^{m_1,m_2} e^{t\Delta_d}f(n)|&=C\left|\sum_{j\in\Z}\delta_{right/left}^{m_1,m_2}G(u,j)|_{u=t/2}\partial_{v}^le^{v\Delta_d}f(n-j)|_{v=t/2} \right|\nonumber\\
    &\le C\|\partial_{v}^le^{v\Delta_d}f|_{v=t/2} \|_\infty\left|\sum_{j\in\Z}\delta_{right/left}^{m_1,m_2}G(u,j)|_{u=t/2}\right|\nonumber\\
   & \le  C t^{-l+\alpha/2}t^{-m/2}.
\end{align}
Now assume that $l<[\alpha/2]+1$ and  $m\in \{1,2\}$ so that $\frac{m}{2}+l\ge [\alpha/2]+1$. Then, $l= [\alpha/2]$ and $m=2.$ Then by using \eqref{derivcomb} with $[\alpha/2]+1$ derivatives in the variable $t$ and the fact that $\partial_t^{[\alpha/2]}\delta_{right/left}^{m_1,m_2} e^{t\Delta_d}f(n)\to 0$ as $t\to \infty$  for each $n\in\Z$ (see Lemma \ref{decay}), we get that 
\begin{align*}
 |\partial_t^l\delta_{right/left}^{m_1,m_2} e^{t\Delta_d}f(n)|&=\left|\partial_t^{[\alpha/2]}\delta_{right/left}^{m_1,m_2} e^{t\Delta_d}f(n)\right|\\
 &=\left|\int_t^\infty \partial_u^{[\alpha/2]+1}\delta_{right/left}^{m_1,m_2} e^{t\Delta_d}f(n)du\right|\le  C t^{-l+\alpha/2}t^{-m/2}.
\end{align*}

\end{proof}

 \section{Proof of  Theorem \ref{caractodos}}\label{caractodo}
 \setcounter{theorem}{0}
\setcounter{equation}{0}

In this section we  are going to prove  Theorem \ref{caractodos}. For that aim we need to prove some results that are important to understand the classes $C^\alpha(\Z),$ $\Lambda^\alpha_H,$ and $\Lambda_P^\alpha.$

\begin{lemma}\label{sizeCalpha}
Let $\alpha>0$, $\alpha\not\in\N,$ and $f\in C^\alpha(\Z)$. Then, there exists a constant $C>0$  such that
$$|f(n)|\le C(1+|n|^\alpha), \quad n\in\Z.
$$
\end{lemma}
\begin{proof}
Assume first that $0<\alpha<1$. Then, $|f(n)|\le |f(n)-f(0)|+|f(0)|\le C (1+|n|^\alpha)$.

Now, assume that $1<\alpha<2$. By definition, this means that $\delta_{right/left}f(n)\in C^{\alpha-1}(\Z)$ and, from the previous case, we have that
$$
|\delta_{right/left}f(n)|\le C(1+|n|^{\alpha-1}).
$$
Therefore, for $n\in\N$,
\begin{align*}
    |f(n)|&\le |f(n)-f(n-1)|+\dots+|f(1)-f(0)|+|f(0)|=\sum_{j=1}^n|\delta_{left}f(j)|+|f(0)|\\
    &\le  C\; n(1+|n|^{\alpha-1})+|f(0)|\le C(1+|n|^\alpha).
\end{align*}
Similarly, for $n\in\Z_-=\Z\setminus\N_0$,
\begin{align*}
    |f(n)|&\le |f(n)-f(n+1)|+\dots+|f(-1)-f(0)|+|f(0)|=\sum_{j=n}^{-1}|\delta_{right}f(j)|+|f(0)|\\
    &\le  C|n|(1+|n|^{\alpha-1})+|f(0)|\le C(1+|n|^\alpha).
\end{align*}
By iterating the previous arguments we get the result for $\alpha>2.$

\end{proof}
The following theorem was proved in \cite[Theorem 4.1]{DL-CT1} for the Hermite operator and \cite[Theorem 5.6]{DL-CT2} for general Schr\"odinger operators satisfying a reverse H\"older inequality. The proof for the discrete Laplacian Lipschitz spaces is the same, so we omit the details.
\begin{theorem}\label{calor-Poisson}
Let $\alpha>0$ and $f:\Z\to\R$ such that $\sum_{j\in\Z}\frac{|f(j)|}{1+|j|^2}<\infty$. If $f\in\Lambda_H^{\alpha}$, then $f\in\Lambda_P^{\alpha}$.
\end{theorem}

\begin{theorem}\label{th1}For $0<\alpha<1$, $C^{\alpha}(\Z)=\Lambda_H^{\alpha}=\Lambda_P^{\alpha}$.
\end{theorem}

\begin{proof}
Let $f\in C^{\alpha}(\Z)$, $0<\alpha<1$. From Lemma \ref{sizeCalpha} we have that $\frac{|f|}{1+|\cdot|^\alpha}\in\ell^\infty(\Z).$


Since  the total mass $\sum_{j\in\Z}G(t,j)=1$, we can write  \begin{align*}
|\partial_t e^{t\Delta_d}f(n)|
&=\left|\sum_{j\in\Z}\partial_t G(t,n-j)(f(j)-f(n))\right|\le C\sum_{j\in\Z}|\partial_t G(t,j)||j|^\alpha.
\end{align*}

{
Since $\partial_t G(t,j)=\Delta_d G(t,j)=\delta_{right}^2G(t,j-1)$, and $\delta_{right}^2G(t,j-1)=\delta_{right}^2G(t,|j|-1)$ for $j\leq -1,$ we can write for every $t>0,$
\begin{align*}
    \sum_{j\in\Z}|\partial_t G(t,j)||j|^\alpha&=2\sum_{j\geq 1}|\delta_{right}^2G(t,j-1)||j|^\alpha=2\biggl(\sum_{1\leq j\leq \sqrt{t}}+\sum_{j> \sqrt{t}}\biggr)|\delta_{right}^2G(t,j-1)||j|^\alpha\\
    &\leq C\biggl(t^{-1+\alpha/2}+\sum_{j> \sqrt{t}}|\delta_{right}^2G(t,j-1)||j|^\alpha\biggr),
\end{align*}
where in the last inequality we have applied Lemma  \ref{kernelest}.
Assume first that $t\leq 1.$ Then, $j>\sqrt{t}$ if and only if $j\geq 1,$ so we have, by \eqref{polyI}, that $$\sum_{j\geq  1}|\delta_{right}^2G(t,j-1)||j|^\alpha\leq C\sum_{j\geq 0}G(t,j)(j+1)^2\leq C(1+p_1(2t))\leq Ct^{-1+\alpha/2}.$$
Now assume that $t>1$. If $j>\sqrt{t}$, then $j\geq 2$ and therefore, $j\le 2(j-1)$. Thus, by using Lemma \ref{heatbounds}, the fact that $G(t,j)$ is decreasing in $j\in\N_0$   and \eqref{polyI}, we get that
\begin{align*}
\sum_{j>\sqrt{t}}|\delta_{right}^2G(t,j-1)||j|^\alpha&\leq\frac{C}{t}\sum_{j>\sqrt{t}}G(t,j-1)\biggl(\frac{(j-1/2)^2}{t}+1\biggr)|j|^{\alpha}\\
&\leq\frac{C}{t^{3-\alpha/2}}\sum_{j\geq 2}G(t,j-1)|j-1|^{4}\leq C\frac{p_2(2t)}{t^{3-\alpha/2}}\leq Ct^{-1+\alpha/2}.
\end{align*}
}

Since for $0<\alpha<1$ a function such that $\frac{|f|}{1+|\cdot|^\alpha}\in\ell^\infty(\Z)$ also satisfies $\sum_{j\in\Z}\frac{|f(j)|}{1+|j|^2}<\infty$, from Theorem \ref{calor-Poisson} we know that $\Lambda_H^{\alpha}\subseteq\Lambda_P^{\alpha}$.

Now we prove that $\Lambda_P^{\alpha}\subseteq C^\alpha(\Z)$.  Let $f\in \Lambda_P^{\alpha}$ and $n\neq m$ integer numbers. We assume without loss of generality that $m>n$. We fix $y=|n-m|>0.$ Then 
\begin{align*}
|f(n)-f(m)|&\leq |f(n)-e^{-y\sqrt{-\Delta_d}} f(n)|+|e^{-y\sqrt{-\Delta_d}}f(n)-e^{-y\sqrt{-\Delta_d}} f(m)|+|e^{-y\sqrt{-\Delta_d}} f(m)-f(m)|\\
&=(I)+(II)+(III).
\end{align*}
From Lemma \ref{semigroupP} B and the hypothesis, we have that $$
(I)= \left|\int_0^y\partial_u e^{-u\sqrt{-\Delta_d}}f(n)\,du\right|\leq C \int_0^y u^{-1+\alpha}\,du=C y^{\alpha}=C|n-m|^{\alpha}.$$ 
The same computation works for $(III).$

On the other hand, by  using Lemma \ref{cambioyx}, we get that
\begin{align*}
|e^{-y\sqrt{-\Delta_d}}f(n)-e^{-y\sqrt{-\Delta_d}} f(m)|&\leq |n-m|\sup_{n'\in [n,m-1]}\left|\delta_{right}e^{-y\sqrt{-\Delta_d}}f(n')\right|\\&\le C|n-m|y^{-1+\alpha}=C|n-m|^\alpha.
\end{align*}
 We conclude that $f\in C^{\alpha}(\Z).$

\end{proof}

\begin{theorem} \label{equiless2}Let $0<\alpha<2$ and $f:\Z\to \C$ be a function such that $\frac{f}{1+|\cdot|^\alpha}\in\ell^\infty(\Z)$ and $\sum_{j\in\Z}\frac{|f(j)|}{1+|j|^2}<\infty$. The following are equivalent:
\begin{itemize}
    \item [(1)]$f\in \Lambda^\alpha_H$. 
        \item [(2)]$f\in \Lambda^\alpha_P$.
        \item [(3)] $f$ satisfies \begin{equation}\label{dif2}
    \sup_{n\neq 0}\frac{\|f(\cdot+n)+f(\cdot-n)-2f(\cdot)\|_\infty}{|n|^\alpha}<\infty.
\end{equation}
\end{itemize} 


\end{theorem}

\begin{proof}From Theorem \ref{calor-Poisson} we know that $(1)\implies(2).$
Let $f\in \Lambda^\alpha_P$. If $0<\alpha<1,$ then from Theorem \ref{th1} we have that
$$
|f(n+m)+f(n-m)-2f(n)|\le |f(n+m)-f(n)|+|f(n-m)-f(n)|\le C|m|^\alpha, \:\; n,m\in\Z.
$$
Now  assume that $1\le \alpha<2$ and, without loss of generality, that $m\in\N$. Then, for $y=m$ and $n\in\Z$ we have 
\begin{align*}
    &|f(n+m)+f(n-m)-2f(n)|\\& \le |f(n+m)- e^{-y\sqrt{-\Delta_d}}f(n+m)+f(n-m)-e^{-y\sqrt{-\Delta_d}}f(n-m)-2f(n)+2e^{-y\sqrt{-\Delta_d}}f(n)|\\
    &\quad+|e^{-y\sqrt{-\Delta_d}}f(n+m)+e^{-y\sqrt{-\Delta_d}}f(n-m)-2e^{-y\sqrt{-\Delta_d}}f(n)|=I+II.
\end{align*}

If $1<\alpha<2,$ Lemmata \ref{semigroupP} B and \ref{cambioyx} gives that
\begin{align*}
    I&=\left|\int_0^y(\partial_u e^{-u\sqrt{-\Delta_d}}f(n+m)+\partial_ue^{-u\sqrt{-\Delta_d}}f(n-m)-2\partial_ue^{-u\sqrt{-\Delta_d}}f(n))du \right|\\
    &\le C\;m\int_0^y\Big(\sup_{n'\in[n,n+m-1]}|\delta_{right}\partial_ue^{-u\sqrt{-\Delta_d}}f(n')|+\sup_{n''\in[n-m,n-1]}|\delta_{right}\partial_ue^{-u\sqrt{-\Delta_d}}f(n'')|\Big)du\\
    &\le C\; m\int_0^y u^{-2+\alpha}du=Cm^\alpha.
\end{align*}
If $\alpha=1$, by using that $\partial_ue^{-u\sqrt{-\Delta_d}}f(n)=-\int_u^y\partial_w^2e^{-w\sqrt{-\Delta_d}}f(n)dw+\partial_ye^{-y\sqrt{-\Delta_d}}f(n),$ we have that
\begin{align*}I&\le C \int_0^y\int_u^y w^{-1}dwdu+\left|\int_0^y(\partial_y e^{-y\sqrt{-\Delta_d}}f(n+m)+\partial_ye^{-y\sqrt{-\Delta_d}}f(n-m)-2\partial_ye^{-y\sqrt{-\Delta_d}}f(n))du \right|\\
&\le C\left(y\log(y)-\int_0^y\log (u) du\right)\\&\quad+|y||m|(\sup_{n'\in[n,n+m-1]}|\delta_{right}\partial_ye^{-y\sqrt{-\Delta_d}}f(n')|+\sup_{n''\in[n-m,n-1]}|\delta_{right}\partial_ye^{-y\sqrt{-\Delta_d}}f(n'')|)\\
&\le C (y+y\;m\;y^{-1})=C\;m.
\end{align*}

On the other hand, we have that 
\begin{align*}
 II&=|(e^{-y\sqrt{-\Delta_d}}f(n+m)-e^{-y\sqrt{-\Delta_d}}f(n+m-1))+\ldots+(e^{-y\sqrt{-\Delta_d}}f(n+1)-e^{-y\sqrt{-\Delta_d}}f(n))\\
 &-(e^{-y\sqrt{-\Delta_d}}f(n)-e^{-y\sqrt{-\Delta_d}}f(n-1))-\ldots -(e^{-y\sqrt{-\Delta_d}}f(n-m+1)-e^{-y\sqrt{-\Delta_d}}f(n-m))|\\
  &=\left| \sum_{j=1}^m (\delta_{right}e^{-y\sqrt{-\Delta_d}}f(n-j)-\delta_{right}e^{-y\sqrt{-\Delta_d}}f(n+j-1))\right|\\
 &\le \sum_{j=1}^m|2j-1|\Big|\sup_{n'\in [n-j,n+j-2]}\delta_{right}(\delta_{right}e^{-y\sqrt{-\Delta_d}}f(n'))\Big|\le Cm^{\alpha}.
 \end{align*}

Finally, we prove  (3)$\implies$(1). 
Suppose that $f$ satisfies \eqref{dif2}. Since $G(t,j)=G(t,-j),$ $j\in\N$, and $\partial_te^{t\Delta_s}1=0,$ we have for $t>0$ that
\begin{align*}
    |\partial_t e^{t\Delta_d}f(n)|&=\left| \frac{1}{2}\sum_{j\in\Z}\partial_t G(t,j) (f(n-j)+f(n+j)-2f(n))\right|\\
    &\le C\sum_{j\in\Z}|\partial_t G(t,j)||j|^\alpha.
\end{align*}
Taking now $t>0,$ the rest of the argument follows as in the proof of Theorem \ref{th1}.

\end{proof}

\begin{remark}\label{dif2calor}
Notice that in the previous theorem the assumption  $\sum_{j\in\Z}\frac{|f(j)|}{1+|j|^2}<\infty$  is only needed in the implications in which $\Lambda_P^\alpha$ appears. It can be proved that $(1)\Longleftrightarrow(3)$ only assuming that the function satisfies   $\frac{f}{1+|\cdot|^\alpha}\in\ell^\infty(\Z)$.
\end{remark}

\begin{theorem}\label{derivada}
Let $\alpha>1$. Then, $f\in\Lambda^\alpha_H$ if, and only if $\delta_{right/left}f\in \Lambda^{\alpha-1}_H$.
\end{theorem}

\begin{proof}
Suppose that  $f\in\Lambda^\alpha_H$ and let $k=[\alpha/2]+1$. 

We prove first that $\frac{|\delta_{right/left}f|}{1+|\cdot|^{\alpha-1}}\in\ell^\infty(\Z)$. Take $n\neq 0.$
From Lemma \ref{semigroupP} A (iii), we have that
\begin{align*}
    |\delta_{right}f(n)|&\le \sup_{0<t<|n|^2}|e^{t\Delta_d}\delta_{right}f(n)|\\
    &\le \sup_{0<t<|n|^2}|e^{t\Delta_d}\delta_{right}f(n)-e^{|n|^2\Delta_d}\delta_{right}f(n)|+|e^{|n|^2\Delta_d}\delta_{right}f(n)|=A+B.
\end{align*}
Regarding $B,$ by using Lemma \ref{semigroupP} A (ii)
 we get that
\begin{align*}
|B|&=|\delta_{right}e^{|n|^2\Delta_d}f(n)|\le C (1+|n|^{\alpha-1}).
\end{align*}

To deal with A we have to distinguish cases. If $1<\alpha<2$, then
$$
|A|= \sup_{0<t<|n|^2}\left|\int_t^{|n|^2}\partial_u\delta_{right}e^{u\Delta_d}f(n)du\right|\le C\sup_{0<t<|n|^2}(|n|^{-1+\alpha}+t^{-1/2+\alpha/2})\le C (1+|n|^{\alpha-1}).
$$
Now consider the case $2\le \alpha<4$, $\alpha\neq 3.$ Then, $[\alpha/2]+1=2$ and from Lemma \ref{cambioyx} we have that

\begin{align}\label{Anoimpar}
    |A|&=\sup_{0<t<|n|^2}\left|\int_t^{|n|^2}\left(\int_u^{|n|^2}\partial_w^2\delta_{right}e^{w\Delta_d}f(n)dw+\partial_v\delta_{right}e^{v\Delta_d}f(n)\Big|_{v=|n|^2}\right)du\right|\nonumber\\
    &\le C\sup_{0<t<|n|^2}\left(\int_t^{|n|^2} \int_u^{|n|^2}w^{-5/2+\alpha/2}dwdu+(|n|^2-t)\partial_v\delta_{right}e^{v\Delta_d}f(n)\Big|_{v=|n|^2}\right)\nonumber\\\
    &\le C \sup_{0<t<|n|^2}\left(\int_t^{|n|^2} (|n|^{-3+\alpha}-u^{-3/2+\alpha/2})du+(|n|^2-t)\partial_v\delta_{right}e^{v\Delta_d}f(n)\Big|_{v=|n|^2}\right)\nonumber\\\
    &\le C \sup_{0<t<|n|^2}\left(|n|^{-3+\alpha}(|n|^2-t)+(|n|^{-1+\alpha}-t^{-1/2+\alpha/2})+(|n|^2-t)\partial_v\delta_{right}e^{v\Delta_d}f(n)\Big|_{v=|n|^2}\right)\nonumber\\\
    &\le C |n|^{\alpha-1}+C|n|^2\partial_v\delta_{right}e^{v\Delta_d}f(n)\Big|_{v=|n|^2}.
\end{align}
Now we use Lemma \ref{semigroupP} A (ii) to get  $$\Big|\partial_v\delta_{right}e^{v\Delta_d}f(n)\Big|_{v=|n|^2}\Big|=|\delta_{right}^3e^{|n|^2\Delta_d}f(n-1)|\leq C\frac{1+|n|^{\alpha}}{n^3}.$$

Therefore,  $|A|\le C (1+ |n|^{\alpha-1}).$
\newline
If $\alpha$ is an even number we can proceed as in \eqref{Anoimpar}, but writing $[\alpha/2]+1$ integrals, such that inside the inner integral will be $\partial_w^{[\alpha/2]+1}\delta_{right}e^{w\Delta_d}f(n)$. 
\newline
If $\alpha$ is odd, we have to proceed similarly, but now it will appear some logarithms in the integrals. We do the case $\alpha=3$ to illustrate the computation,  but the rest of cases are analogous.
\begin{align*}
    |A|&=\sup_{0<t<|n|^2}\left|\int_t^{|n|^2}\left(\int_u^{|n|^2}\partial_w^2\delta_{right}e^{w\Delta_d}f(n)dw+\partial_v\delta_{right}e^{v\Delta_d}f(n)\Big|_{v=|n|^2}\right)du\right|\nonumber\\
    &\le\sup_{0<t<|n|^2}\left(\int_t^{|n|^2} \int_u^{|n|^2}w^{-1}dwdu+(|n|^2-t)\partial_v\delta_{right}e^{v\Delta_d}f(n)\Big|_{v=|n|^2}\right)\nonumber\\\
    &\le C\sup_{0<t<|n|^2}\left(\int_t^{|n|^2} (\log(|n|^2)-\log u)du+(|n|^2-t)\partial_v\delta_{right}e^{v\Delta_d}f(n)\Big|_{v=|n|^2}\right)\nonumber\\\
    &= C\sup_{0<t<|n|^2}[\log|n|^2(|n|^2-t)-(|n|^2\log|n|^2)+|n|^2+t\log t-t+(|n|^2-t)\partial_v\delta_{right}e^{v\Delta_d}f(n)\Big|_{v=|n|^2}]\nonumber\\\
    &\le C |n|^2+C|n|^2\partial_v\delta_{right}e^{v\Delta_d}f(n)\Big|_{v=|n|^2}\le C (1+|n|^2).
\end{align*}

Now we prove the condition on the semigroup. Lemma \ref{cambioyx} implies that 
$$
\|\partial_t^k\delta_{right/left}e^{t\Delta_d}f\|_\infty \le C t^{-(k+1/2)+\alpha/2}=Ct^{-k+\frac{\alpha-1}{2}}.
$$
Since $\partial_t^k\delta_{right/left}e^{t\Delta_d}f=\partial_t^ke^{t\Delta_d}(\delta_{right/left}f)$, see Remark \ref{obs}, from Lemma \ref{subirk} we get that $\delta_{right/left}f\in \Lambda_H^{\alpha-1}$.

Assume now that $\delta_{right/left}f\in \Lambda_H^{\alpha-1}$. By definition, we have that $\frac{|\delta_{right/left}f|}{(1+|\cdot|^{\alpha-1})}\in \ell^\infty(\Z).$ Thus, the proof of Lemma \ref{sizeCalpha} gives that $\frac{|f|}{(1+|\cdot|^{\alpha})}\in \ell^\infty(\Z).$




Let $k=[(\alpha-1)/2]+1$. From Lemma \ref{cambioyx} we have that
$$
\|\partial_t^k\delta_{left/right}e^{t\Delta_d}(\delta_{right/left}f)\|_\infty \le C t^{-(k+1/2)+\frac{\alpha-1}{2}}=Ct^{-(k+1)+\frac{\alpha}{2}}.
$$
Since $\partial_t^k\delta_{left/right}e^{t\Delta_d}(\delta_{right/left}f)=\partial_t^k\delta_{left/right}\delta_{right/left}e^{t\Delta_d}f$, we have that
$$
\|\partial_t^k\Delta_de^{t\Delta_d}f\|_\infty \le Ct^{-(k+1)+\frac{\alpha}{2}}.
$$
Therefore, \eqref{eq1} gives that $\|\partial_t^{k+1}e^{t\Delta_d}f\|_\infty \le C t^{-(k+1)+\alpha/2},$
so from  Lemma \ref{subirk} we conclude that  $f\in\Lambda^\alpha_H$.

\end{proof}

\begin{theorem}\label{derivadaP}
        Let $\alpha>1$ and $f:\Z\to\R$. If $f\in\Lambda^\alpha_P$, then $\delta_{right/left}f\in \Lambda^{\alpha-1}_P$.
\end{theorem}

\begin{proof}
Let $k=[\alpha].$ Suppose that $f\in\Lambda^\alpha_P$. Then $\sum_{j\in\Z}\frac{|f(j)|}{1+|j|^2}<\infty$ and
Lemma \ref{cambioyx} implies that 
$$
\|\partial_y^k\delta_{right/left}e^{-y\sqrt{-\Delta_d}}f\|_\infty \le C y^{-(k+1)+\alpha}=Cy^{-k+{\alpha-1}}.
$$
It is clear that 
$
\sum_{j\in\Z}\frac{|\delta_{right}f(j)|}{1+|j|^2}<\infty.$ 
Moreover, since $\partial_y^k\delta_{right/left}e^{-y\sqrt{-\Delta_d}}f=\partial_y^ke^{-y\sqrt{-\Delta_d}}(\delta_{right/left}f)$, see Remark \ref{obs}, from Lemma \ref{subirk} we get that $\delta_{right/left}f\in \Lambda^{\alpha-1}_P$.

\end{proof}

Finally, we can prove Theorem \ref{caractodos}.
\vspace{0.4 cm}

{\it Proof of Theorem \ref{caractodos}}. We prove first (A.1). In Theorem \ref{th1} 
 we have proved the result for $0<\alpha<1$. 
Let $k< \alpha<k+1$, for certain $k\in\N.$ Assume first that $f\in \Lambda_H^\alpha $. Then, by applying $k$ times Theorem \ref{derivada} we get that $\delta_{right/left}^{l,s}f\in\Lambda_H^{\alpha-k}$, $l+s=k$, and from Theorem \ref{th1} and the definition of $C^{\alpha-k}(\Z)$ we get that
$$
\sup_{n\neq m}\frac{|\delta_{right/left}^{l,s}f(n)-\delta_{right/left}^{l,s}f(m)|}{|n-m|^{\alpha-k}}<\infty, \:\; \text{ whenever }l+s=k,
$$
so $f\in C^\alpha(\Z).$

Conversely, suppose that $f\in C^\alpha(\Z).$ From Lemma \ref{sizeCalpha} we know that $\frac{|f|}{1+|\cdot|^\alpha}\in\ell^\infty(\Z).$ Moreover, the definition of the space gives that $\delta_{right/left}^{l,s}f \in C^{\alpha-k}(\Z)$, $l+s=k$. Therefore, Theorem \ref{th1} implies that  $\delta_{right/left}^{l,s}f\in \Lambda_H^{\alpha-k}$, $l+s=k$. Applying k times Theorem \ref{derivada}, we conclude that $f\in\Lambda^\alpha_H.$

Regarding the proof of (A.2), we proceed as in the proof of (A.1) but now we use Theorem \ref{equiless2} (see Remark \ref{dif2calor}) instead of Theorem \ref{th1}.

In virtue of Theorem \ref{calor-Poisson} and (A.1), to establish (B) we only need to prove that $f\in \Lambda_P^\alpha\implies f\in C^\alpha(\Z).$ Let $f\in \Lambda_P^\alpha$. 
Then, by applying $k$ times Theorem \ref{derivadaP} we get that $\delta_{right/left}^{l,s}f\in\Lambda_P^{\alpha-k}$, $l+s=k$, and from Theorem \ref{th1} and the definition of $C^{\alpha-k}(\Z)$ we get that
$$
\sup_{n\neq m}\frac{|\delta_{right/left}^{l,s}f(n)-\delta_{right/left}^{l,s}f(m)|}{|n-m|^{\alpha-k}}<\infty, \:\; \text{ whenever }l+s=k,
$$
so $f\in C^\alpha(\Z).$

Regarding the proof of (B.2), we proceed as in the proof of (B.1) but now we use Theorem \ref{equiless2} instead of Theorem \ref{th1}.
\edproof

\section{Applications}\label{applications}
\setcounter{theorem}{0}
\setcounter{equation}{0}

In this section we shall prove regularity results for fractional powers of the discrete Laplacian in the Lipschitz spaces defined through the heat semigroup. To that aim, we recall the definition of the fractional powers of the discrete Laplacian, by using the semigroup method, see \cite{CRSTV,Stinga,ST}. For other works considering fractional powers of the discrete Laplacian see for instance \cite{Miana, LR}.


  Let $I$ denote the identity operator. For good enough functions, we define the following operators:
  \begin{itemize}\item The Bessel potential of order $\beta>0$, 
  $$
  (I-\Delta_d)^{-\beta/2} f(n)=\frac{1}{\Gamma(\beta/2)}\int_0^\infty e^{-\tau (I-\Delta_d)}f(n)\tau^{\beta/2}\frac{d\tau}{\tau},  \:\:n\in\Z.
  $$
  \item The positive fractional power of the Laplacian, \begin{equation}\label{positivepower}
  (-\Delta_d)^\beta f(n)=\frac{1}{c_{\beta}}\int_0^\infty\left(e^{\tau \Delta_d}-I\right)^{[\beta]+1}f(n) \frac{d\tau}{\tau^{1+\beta}}, \:\:n\in\Z, \quad \beta>0,
  \end{equation}
  where $c_{\beta}=\int_0^\infty\left( e^{-\tau}-1\right)^{[\beta]+1} \frac{d\tau}{\tau^{1+\beta}}$.
   \item The negative fractional power of the Laplacian,
   $$
   (-\Delta_d)^{-\beta}f(n)=\frac{1}{\Gamma(\beta)}\int_0^\infty e^{\tau\Delta_d}f(n)\frac{d\tau}{\tau^{1-\beta}},  \:\:n\in\Z, \quad 0<\beta<1/2.
  $$

  \end{itemize}
  The previous formulae come from the following Gamma formulae, see \cite{CRSTV},
  	\begin{equation}\label{fractpow}
  	\lambda^{-\beta} = \frac1{\Gamma(\beta)} \int_0^\infty e^{-\lambda t }t^{\beta}\,\frac{dt}{t},\quad
  	\hbox{and}\quad \lambda^{\beta} = \frac1{c_\beta} \int_0^\infty
  	(e^{-\lambda t}-1)^{[\beta]+1}\,\frac{dt}{t^{1+\beta}},\end{equation}
  	where $ \beta >0 $ and $\lambda$ is a complex number with $\Real \lambda \ge 0$.
  	
 As it will be shown in Theorem \ref{Besselpot}, Bessel potentials of order $\beta>0$ are well defined for $f\in\Lambda^\alpha_H$, $\alpha>0$. However, the fractional powers of the Laplacian, $(-\Delta_d)^{\pm\beta}$,  are not well defined in general for $\Lambda_H^\alpha$ functions and an additional condition is needed. In \cite{CRSTV}, the authors assumed that the functions belongs to the space
$$
\ell_{\pm\beta}:=\left\{u: \Z\to\R: \:\; \sum_{m\in\Z}\frac{|u(m)|}{(1+|m|)^{1\pm 2\beta}}<\infty\right\},
$$
 in order to define $(-\Delta_d)^{\pm\beta}f,$ where $0<\beta<1$ in the case of the positive powers and $0<\beta<1/2$ for the negative ones. Note that such spaces are the analogs in the discrete setting of the ones  considered in \cite{Silvestre} for the Laplacian in $\R^n$. The choice of these spaces is justified since  the discrete kernel in the pointwise formula 
\begin{equation}\label{potenciapuntual}
(-\Delta_d)^{\pm \beta}f(n)=\sum_{m\in\Z}K_{\pm \beta}(n-m)f(m),  \:\:n\in\Z,
\end{equation}
satisfies 
$K_{\beta}(m)\sim \frac{1}{|m|^{1+2\beta}},$
whenever $0<\beta<1$ and $K_{-\beta}(m)\sim \frac{1}{|m|^{1- 2\beta}}$, for $0<\beta<1/2$, see \cite{CRSTV}.
Observe that the negative powers of the Laplacian  are only well defined for  $0<\beta<1/2,$ since the integral that defines it is not absolutely convergent for $\beta\ge 1/2,$ see \eqref{asymptotic2}.

In this section, we also want to prove regularity results for positive powers than can be bigger that 1. For that purpose we extend the definition above of $\ell_\beta$ for any $\beta> 0.$ Let $\beta>-1/2$ and $n\in\Z,$ we define the discrete kernel \begin{equation}\label{FractKernel}
\displaystyle K_{\beta}(n):=\left\{\begin{array}{ll}
0,&|n|-\beta-1\in\N_0,\\ \\
\displaystyle\frac{(-1)^{|n|}\Gamma(2\beta+1)}{\Gamma(1+\beta+|n|)\Gamma(1+\beta-|n|)},&\text{otherwise}.
\end{array} \right.
\end{equation}
Note that when $\beta\in\N_0,$ then $K_{\beta}(n)=0$ for all $|n|\geq \beta +1.$

\begin{lemma}\label{lbeta}
 Let $f\in\ell_\beta$, $\beta>0$. Then, $(-\Delta_d)^\beta f$ is well defined and
 $$
 |(-\Delta_d)^{\beta}f(n)|\le C\sum_{j\in\Z}\frac{|f(j)|}{1+|n-j|^{1+2\beta}},\quad n\in\Z.
 $$
 Moreover, in that case, $$(-\Delta_d)^{\beta}f(n)=\sum_{j\in\Z}K_{\beta}(j)(f(n-j)-f(n)),\quad n\in\Z.$$
\end{lemma}
\begin{proof}
First note that since $f\in \ell_{\beta},$ $f$ has polynomial growth and then $e^{t\Delta_d} f$ is well defined. Let $k\in\N$ such that $k-1\leq \beta<k$ (so $k=[\beta]+1$). Then 
\begin{align*}
(e^{t\Delta_d}-I)^k f(n)&=\sum_{l=0}^k(-1)^{k-l}\binom{k}{l}e^{tl\Delta_d}f(n)\\
&=\sum_{l=1}^k (-1)^{k-l}\binom{k}{l}\biggl(\sum_{j\in\N}G(lt,j)(f(n+j)+f(n-j))+G(lt,0)f(n)\biggr)+(-1)^kf(n).
\end{align*} 

Since $-1=\sum_{l=1}^k(-1)^l\binom{k}{l}$ and $G(lt,0)-1=-2\sum_{j\in\N}G(lt,j),$ one obtains that
\begin{align*}
(-1)^kf(n)\left( \sum_{l=1}^k (-1)^{l}\binom{k}{l}G(lt,0)-\sum_{l=1}^k (-1)^{l}\binom{k}{l}\right)=(-1)^kf(n)\sum_{l=1}^k (-1)^{l}\binom{k}{l}(-2\sum_{j\in\N}G(lt,j))
\end{align*}
and therefore 
\begin{align*}
(e^{t\Delta_d}-I)^kf(n) &= \sum_{j\in\N}(f(n+j)+f(n-j)-2f(n))\sum_{l=1}^k(-1)^{k-l}\binom{k}{l}G(lt,j)\\
&=\sum_{j\in\Z}(f(n-j)-f(n))\sum_{l=1}^k(-1)^{k-l}\binom{k}{l}G(lt,j)\\
&=\sum_{j\in\Z\setminus\{0\}}(f(n-j)-f(n))\sum_{l=1}^k(-1)^{k-l}\binom{k}{l}G(lt,j).
\end{align*}

Now we denote $$T(t,j):=\sum_{l=1}^k(-1)^{k-l}\binom{k}{l}G(lt,j)=\sum_{l=0}^k(-1)^{k-l}\binom{k}{l}G(lt,j),\quad j\in\Z\setminus\{0\},$$ where in the last identity we have used that $G(0,j)=0$ for $j\neq 0.$

From \eqref{fourier} and the fact that $\sum_{l=0}^k\int_{-\pi}^\pi \left|\binom{k}{l}e^{-ij\theta}e^{-4lt\sin^2\theta/2}\right|\,d\theta<\infty$, we can apply Fubini's theorem to get that $$T(t,j)=\frac{1}{2\pi}\int_{-\pi}^\pi e^{-ij\theta}(e^{-4t\sin^2 \theta/2}-1)^k\,d\theta.$$

By \eqref{fractpow}, observe that for all $j\neq 0$ $$\int_0^{\infty}|T(t,j)|\frac{dt}{t^{1+\beta}}\leq C\int_{-\pi}^\pi\int_0^{\infty} (1-e^{-4t\sin^2\theta/2})^k\frac{dt}{t^{1+\beta}}\,d\theta\leq C\int_{-\pi}^{\pi}(\sin^2\theta/2)^\beta\,d\theta<\infty,$$ and therefore \begin{align*}
\frac{1}{c_{\beta}}\int_0^{\infty}T(t,j)\frac{dt}{t^{1+\beta}}&=\frac{4^{\beta}}{2\pi}\int_{-\pi}^{\pi}e^{-ij\theta}(\sin^2 \theta/2)^\beta\,d\theta=\frac{ 4^{\beta}}{\pi}\int_{-\pi}^{0}\cos(j\theta)(\sin^2 \theta/2)^\beta\,d\theta\\
&=\frac{2}{\pi}4^{\beta}\int_{-\pi/2}^{0}\cos(2j\theta)(\sin^2 \theta)^\beta\,d\theta=\frac{2}{\pi}4^{\beta}(-1)^j\int_{0}^{\pi/2}\cos(2j\theta)\cos^{2\beta} \theta\,d\theta\\
&=K_{\beta}(j),
\end{align*}
{see \cite[Section 2.5.12, formula (22)]{PBM}}.

Finally, for $|j|\geq k$ we have by \eqref{IntFractBessel} that $$\int_0^{\infty}|T(t,j)|\frac{dt}{t^{1+\beta}}\leq C\sum_{l=1}^k\int_0^{\infty}G(lt,j)\frac{dt}{t^{1+\beta}}\leq \frac{C}{1+|j|^{1+2\beta}}.$$

Therefore, we have proved that $(-\Delta_d)^{\beta}f$ is well-defined,  that $$
|(-\Delta_d)^{\beta}f(n)|\leq C\sum_{j\in\Z}|f(n-j)-f(n)|\frac{1}{1+|j|^{1+2\beta}}\leq C\sum_{j\in\Z}\frac{|f(j)|}{1+|n-j|^{1+2\beta}},
$$
and that $(-\Delta_d)^{\beta}f(n)=\sum_{j\in\Z}(f(n-j)-f(n))K_{\beta}(j).$



\end{proof}

\begin{remark}\label{nucleoderiv}
Some observations are now in order:
\begin{itemize}
    \item Note that if $\beta \in\N_0,$ the definition of $K_{\beta}$ (see \eqref{FractKernel}) implies that $K_{\beta}$ is a sequence of compact support, so $K_{\beta}$ belongs to $\ell^1(\Z)$. Also, if $\beta>0$ is not  a natural number, then the proof above gives that $|K_{\beta}(j)|\leq \frac{C}{1+|j|^{1+2\beta}}$ for all $j\in\Z.$ So $K_{\beta}\in\ell^1(\Z)$ for all $\beta\geq 0.$ Moreover,  in the previous proof one also have that $K_{\beta}(0)=\frac{4^{\beta}}{2\pi}\int_{-\pi}^{\pi}(\sin^2 \theta/2)^\beta\,d\theta,$ so $K_{\beta}(j)$ are the Fourier coefficients of the function $(2-z-1/z)^{\beta}=(4\sin^2\theta/2)^{\beta},$ $z=e^{i\theta}\in\mathbb{T}.$ Taking $z=1$ we get $$\sum_{j\in\Z}K_{\beta}(j)=0,$$ so if $f\in\ell_{\beta}$, then 
    $$
    (-\Delta_d)^{\beta}f(n)=\sum_{j\in\Z}K_{\beta}(j)f(n-j).
    $$
    \item Lemma \ref{lbeta} extends and complements \cite[Theorem 1.1 (i) and Theorem 1.3 (i)]{CRSTV}.
    \item When $\beta$ is a natural number, the expression $(-\Delta_d)^\beta f(n)=\frac{1}{c_{\beta}}\int_0^\infty\left(e^{\tau \Delta_d}-I\right)^{[\beta]+1}f(n) \frac{d\tau}{\tau^{1+\beta}}$ given at the beginning of this section coincides with the classical power $(-\Delta_d)^{\beta}f(n)$ whenever $f\in\ell_\beta$ (recall that any power of $\Delta_d f$ is defined for every sequence $f$).
    \end{itemize}
\end{remark}

\begin{lemma}\label{sequeda}
Let $f:\Z\to\R$.
\begin{itemize}
\item If $f\in{\ell}_{ -\beta}$, $0<\beta<1/2,$ then for every $s>0,$ $e^{s\Delta_d}f\in{\ell}_{-\beta}.$
\item If $f\in{\ell}_{ \beta}$, $\beta>0,$ then for every $s>0,$ $e^{s\Delta_d}f\in{\ell}_{\beta}.$
\end{itemize}
\end{lemma}
\begin{proof}

Suppose that $f\in{\ell}_{ -\beta}$, for some $0<\beta<1/2$ and let $s>0$.
Then,
\begin{align*}
    \sum_{m\in\Z}\frac{|e^{s\Delta_d}f(m)|}{1+|m|^{1-2\beta}}&\le \sum_{m\in\Z}\frac{\sum_{j\in\Z}G(s,m-j)|f(j)|}{1+|m|^{1-2\beta}}=\frac{\sum_{j\in\Z}|f(j)|\sum_{u\in\Z}G(s,u)}{1+|j+u|^{1-2\beta}}\\
     &\le \sum_{j\in\N}\left( \sum_{u=-\infty}^{-(j+1)}+\sum_{u=-j}^{-1}+\sum_{u=0}^\infty\right)\frac{G(s,u)}{1+|j+u|^{1-2\beta}}|f(j)|\\
    &+ \sum_{j\in\Z_-} \sum_{u\in\Z}\frac{G(s,u)}{1+|j+u|^{1-2\beta}}|f(j)|+\sum_{u\in\Z}\frac{G(s,u)}{1+|u|^{1-2\beta}}|f(0)|.\\
\end{align*}
Observe that the last sum is clearly bounded. On the other hand, 
$$
\sum_{j\in\N}\sum_{u=0}^\infty\frac{G(s,u)}{1+|j+u|^{1-2\beta}}|f(j)|\le \sum_{j\in\N}\frac{|f(j)|}{1+|j|^{1-2\beta}}\sum_{u=0}^\infty{G(s,u)}\le C<\infty.
$$
Now we split the sum in $j\in\N$ into two, obtaining
\begin{align*}
   & \sum_{j=1}^{[\sqrt{s}]+1}|f(j)|\sum_{u=j+1}^\infty \frac{G(s,u)}{1+(u-j)^{1-2\beta}}+ \sum_{j=1}^{[\sqrt{s} ]+1}|f(j)|\sum_{u=1}^j \frac{G(s,u)}{1+(j-u)^{1-2\beta}}\\
    &\qquad\le \sum_{j=1}^{[\sqrt{s}]+1}|f(j)|\sum_{u=1}^\infty{G(s,u)}\le C_s
\end{align*}
and, by using \eqref{polyI},
\begin{small}
\begin{align*}
   \sum_{j=[\sqrt{s}]+1}^\infty\sum_{u=j+1}^\infty& \frac{G(s,u)|f(j)|}{1+(u-j)^{1-2\beta}} + \sum_{j=[\sqrt{s} ]+1}^\infty\sum_{u=1}^{[j/2]} \frac{G(s,u)|f(j)|}{1+(j-u)^{1-2\beta}}+\sum_{j=[\sqrt{s} ]+1}^\infty\sum_{u=[j/2]+1}^j \frac{G(s,u)|f(j)|}{1+(j-u)^{1-2\beta}}\\
   &\le  \sum_{j=[\sqrt{s}]+1}^\infty|f(j)|\sum_{u=j+1}^\infty \frac{G(s,u)}{1+u^{1-2\beta}} (1+u^{1-2\beta})+ \sum_{j=[\sqrt{s} ]+1}^\infty\sum_{u=1}^{[j/2]} \frac{G(s,u)|f(j)|}{1+(j/2)^{1-2\beta}}\\
   &\qquad+\sum_{j=[\sqrt{s} ]+1}^\infty \frac{|f(j)|}{\left(\frac{1}{2}\right)^{1-2\beta}+\left(\frac{j}{2}\right)^{1-2\beta}}\sum_{u=[j/2]+1}^j G(s,u)\left( \left(\frac{1}{2}\right)^{1-2\beta}+\left(\frac{j}{2}\right)^{1-2\beta}\right)\\
    &\le C\sum_{j=[\sqrt{s}]+1}^\infty\frac{|f(j)|}{1+j^{1-2\beta}}\sum_{u=1}^\infty{G(s,u)}(1+u^{1-2\beta})
    \le C(1+p_k(2s)),
\end{align*}
\end{small}
where $k$ is the least natural number such that $1-2\beta<2k$. 
For the sum with $j\in\Z_-$ we can proceed similarly. We left the details to the interested reader.

Now assume that $f\in\ell_\beta$, for some $\beta>0. $ Then we can proceed in a completely analogous way as in the previous case, but now the power will be $1+2\beta$, instead of $1-2\beta$.
\end{proof}

Now we prove our main results of this section.
\vspace{0.4 cm}
  
   {\it Proof of Theorem \ref{Besselpot}.}
   Let $f\in \Lambda_H^{\alpha}$ and $\ell=[\frac{\alpha+\beta}{2}]+1$.
    From Lemma \ref{semigroupP} we have that 
    \begin{align*}
| ({I}-\Delta)^{-\beta/2} f(n)|&\le C  \int_0^\infty e^{-\tau} (1+|n|^\alpha+\tau^{\alpha/2})\tau^{\beta/2}\frac{d\tau}{\tau}      \\
&\le C (1+|n|^\alpha)\int_0^\infty e^{-\tau} (1+\tau^{\alpha/2})\tau^{\beta/2}\frac{d\tau}{\tau} \le  C (1+|n|^{\alpha+\beta}),\quad n\in\Z.
    \end{align*}
  Now we prove the condition on the semigroup.  By using again Lemma \ref{semigroupP}  we obtain that  \begin{align*}|
\partial_te^{t\Delta_d}f(n)|&=|\Delta_de^{t\Delta_d}f(n)|=|e^{t\Delta_d}f(n+1)+e^{t\Delta_d}f(n-1)-2e^{t\Delta_d}f(n)|\\
&\le C (1+|n|^\alpha+t^{\alpha/2}), \:\; n\in\Z, \:\;t>0.
\end{align*}
Thus,
\begin{align*}|
\partial_t^2e^{t\Delta_d}f(n)|&=|\partial_t(\Delta_de^{t\Delta_d}f(n))|=|\partial_te^{t\Delta_d}f(n+1)+\partial_te^{t\Delta_d}f(n-1)-2\partial_te^{t\Delta_d}f(n)|\\
&\le C(1+|n|^\alpha+t^{\alpha/2}), \:\; n\in\Z, \:\;t>0.
\end{align*}
By iterating the arguments, we have that $|\partial_t^\ell e^{t\Delta_d}f(n)|\le C (1+|n|^\alpha+t^{\alpha/2}), \:\; n\in\Z, \:\;t>0.$
   
   Therefore,  can introduce the derivatives inside the integral and by using Remark \ref{obs} and Lemma \ref{subirk} we have that 
    \begin{align*}
    |\partial_y^\ell e^{y\Delta_d} (({I}-\Delta)^{-\beta/2} f)(n)|&=\left|\frac{1}{\Gamma(\beta/2)}\int_0^\infty e^{-\tau }\partial_y^\ell e^{y\Delta_d}(e^{\tau\Delta_d} f)(n) \tau^{\beta/2}\frac{d\tau}{\tau}\right|\\
    &\le  C_\beta\int_0^\infty e^{-\tau}|(\partial_w^\ell e^{w\Delta_d}f(n)\Big|_{w=y+\tau})|\tau^{\beta/2}\frac{d\tau}{\tau}
    \\&\le  C_\beta\int_0^\infty e^{-\tau}(y+\tau)^{-\ell+\alpha/2}\tau^{\beta/2}\frac{d\tau}{\tau}
    \\&\stackrel{\frac{\tau}{y}=u}{\le } C_\beta {y^{\alpha/2+\beta/2-\ell}}\int_0^\infty \frac{u^{\beta/2}e^{-yu}}{(1+u)^{\ell-\alpha/2}}\frac{du}{u}
    \\&\le  C_\beta y^{\alpha/2+\beta/2-\ell}.
    \end{align*}
    
    When $f\in \ell^\infty(\Z)$ we proceed analogously by using that, for $\ell= [\beta/2]+1$, $$\|\partial_u^\ell e^{u\Delta_d}  f\|_\infty \le\sup_{n\in\Z}\sum_{j\in\Z}|\partial_u^\ell G(u,j)||f(n-j)|\le C \frac{\|f\|_{\infty}}{u^{\ell}}, \qquad u>0,$$ 
    see Lemma \ref{kernelest} and Remark \ref{Remkernelest}.
    \edproof
 \vspace{0.4 cm}

     

{\it Proof of Theorem \ref{teoSchau}.} We shall prove that if $f\in 	\Lambda_H^{\alpha}\cap \ell_{-\beta}$, then $$\frac{|(-\Delta)^{-\beta}f|}{1+|\cdot|^{\alpha+2\beta}}\in\ell^\infty(\Z).$$
Let $f\in 	\Lambda_H^{\alpha}\cap \ell_{-\beta}$. Since \eqref{potenciapuntual} holds, from  we have that 
\begin{align*}
    |(-\Delta)^{-\beta}f(n)|\le \sum_{j\in\Z}\frac{|f(j)|}{1+|n-j|^{1-2\beta}}= \sum_{|n-j|>2|n|}\frac{|f(j)|}{1+|n-j|^{1-2\beta}}+\sum_{|n-j|\le 2|n|}\frac{|f(j)|}{1+|n-j|^{1-2\beta}}.
\end{align*}
Since $|n-j|\ge \frac{|j|}{2}$ when $|n-j|>2|n|,$ by using that $f\in \ell_{-\beta}$, we get that the first summand is bounded.

On the other hand, by using that $\frac{|f|}{1+|\cdot|^\alpha}\in\ell^\infty(\Z)$ and $|j|\le 3|n|$ when $|n-j|\le 2|n|$,  we have that
\begin{align*}
    \sum_{|n-j|\le 2|n|}\frac{|f(j)|}{1+|n-j|^{1-2\beta}}\le C (1+|n|^\alpha)\sum_{|n-j|\le 2|n|}\frac{1}{1+|n-j|^{1-2\beta}}\le C(1+|n|^{\alpha+2\beta}).
\end{align*}

Following the same steps, it can be proved that  for $f\in \ell^\infty(\Z)\cap \ell_{-\beta}$, we have that $\frac{|(-\Delta)^{-\beta}f|}{1+|\cdot|^{2\beta}}\in\ell^\infty(\Z).$

Let $n\in\Z.$ From Lemma \ref{sequeda} we know that, for every $y>0,$ $e^{y\Delta_d}f\in\ell_{-\beta}$. Moreover, since $\partial_ye^{y\Delta_d}g(n)=\Delta_de^{y\Delta_d}g(n)$ we can introduce the derivatives inside the integral and apply Fubini's theorem so that, for every $\ell\in\N,$
$$
|\partial_y^\ell e^{y\Delta_d} ((-\Delta)^{-\beta} f)(n)|=\left|\frac{1}{\Gamma(\beta)}\int_0^\infty \Delta_d^\ell e^{\tau\Delta_d}(e^{y\Delta_d} f)(n) \tau^{\beta}\frac{d\tau}{\tau}\right|<\infty.
$$

The rest of the proof of $\displaystyle\|\partial_y^\ell e^{y\Delta_d} ((-\Delta)^{-\beta} f)\|_\infty\le Cy^{-\ell+{\alpha+2\beta}}$, $\ell=[\alpha+2\beta]+1$, follows the same steps as the corresponding proof on Theorem \ref{Besselpot}. 
\edproof

 \vspace{0.4 cm}
 \newpage
 
 {\it Proof of Theorem \ref{teoHolder}.}
 
 We prove first (i). Let $f\in 	\Lambda_H^{\alpha}\cap \ell_{\beta}$, $\alpha>2\beta.$ Then, by proceeding in a completely analogous way as in Theorem \ref{teoSchau}, but now the power will be $2\beta$, instead of $-2\beta$, we get that $$\frac{|(-\Delta)^{\beta}f|}{1+|\cdot|^{\alpha-2\beta}}\in\ell^\infty(\Z).$$



 Now we prove the condition on the semigroup.
 
Let $n\in\Z$ and  $\ell=[\beta]+1$. From Lemma \ref{sequeda} we know that, for every $y>0,$ $e^{y\Delta_d}f\in\ell_{\beta}$. Moreover, since $\partial_ye^{y\Delta_d}g(n)=\Delta_de^{y\Delta_d}g(n)$ we can introduce the derivatives inside the integral and apply Fubini's theorem so that, for every $m\in\N,$
\begin{align*}
\Big|\partial_y^me^{y\Delta_d}( (-\Delta_d)^{{\beta}} f)&(n) \Big|= \Big|\frac{1}{c_\beta} \partial_y^me^{y\Delta_d}\Big( \int_0^\infty \int_0^t \underbrace{\dots}_{\substack{\ell}} \int_0^t \partial_\nu^\ell e^{\nu\Delta_d} |_{\nu= s_1+\dots+s_\ell} f(n)ds_1\dots ds_\ell  \frac{dt}{t^{1+\beta}} \Big)\Big|\\
&=	  	\Big|C\int_0^\infty \Big( \int_0^t \underbrace{\dots}_{\substack{\ell}} \int_0^t \partial_\nu^{m+\ell}e^{\nu\Delta_d} |_{\nu=y+ s_1+\dots+s_\ell}f(n) ds_1\dots ds_\ell \Big) \frac{dt}{t^{1+\beta}}\Big|\\&=	  	\Big|C\int_0^\infty \Big( \int_0^t \underbrace{\dots}_{\substack{\ell}} \int_0^t \Delta_d^{m+\ell}e^{\nu\Delta_d} |_{\nu=y+ s_1+\dots+s_\ell}f(n) ds_1\dots ds_\ell \Big) \frac{dt}{t^{1+\beta}}\Big|<\infty.
\end{align*}
Let $m= \left[\frac{\alpha}{2}-\beta\right]+1$. Then, $m+\ell=\left[\frac{\alpha}{2}-\beta\right]+1+[\beta]+1 >\alpha/2-\beta+\beta=\alpha/2$. As $m+\ell\in \mathbb{N}$ we get $m+\ell \ge[\alpha/2]+1.$ Therefore, by using Lemma \ref{subirk}, we get that
	  	\begin{align*}
	  	\Big|\partial_y^me^{y\Delta_d}( (-\Delta_d)^{{\beta}} f)(n) \Big|
	  	&=
	  	\Big|C\int_0^\infty \Big( \int_0^t \underbrace{\dots}_{\substack{\ell}} \int_0^t \partial_\nu^{m+\ell}e^{\nu\Delta_d} |_{\nu=y+ s_1+\dots+s_\ell}f(n) ds_1\dots ds_\ell \Big) \frac{dt}{t^{1+\beta}}\Big|
	  	\\ &\le
	  	C \int_0^\infty \Big( \int_0^t \underbrace{\dots}_{\substack{\ell}} \int_0^t (y+s_1+\dots s_\ell)^{-(m+\ell) +\alpha/2} ds_1\dots ds_\ell \Big) \frac{dt}{t^{1+\beta}} \\&=
	  	C \int_0^y ( \dots )  \frac{dt}{t^{1+\beta}}  + C \int_y^\infty  (\dots ) \frac{dt}{t^{1+\beta}} =C\,[(I) +(II)].
	  	\end{align*}
	   Now we shall estimate $(I)$ and $(II)$.
		  	\begin{align*} 
	  	(I)  &= 
	  	C y^{-m+\alpha/2} \int_0^y  \int_0^{t/y} \underbrace{\dots}_{\substack{\ell}} \int_0^{t/y} (1+s_1+\dots s_\ell)^{-(m+\ell) +\alpha/2} ds_1\dots ds_\ell \frac{dt}{t^{1+\beta}}  \\&\le 
	  	C\, y^{-m+\alpha/2} \int_0^y   \Big(\frac{t}{y}\Big) ^\ell \frac{dt}{t^{1+\beta}}
	  	= C\,y^{-m+\alpha/2-\beta}.
	  	\end{align*}
	   On the other hand,
	  	\begin{align*}
	  	(II)  &\le  \int_y^\infty  \sum_{j=0}^\ell \frac{C_j}{(y+jt)^{m-\alpha/2}} \frac{dt}{t^{1+\beta}} =    \sum_{j=0}^\ell \int_y^\infty \frac{C_j}{(y+jt)^{m-\alpha/2}} \frac{dt}{t^{1+\beta}} \le   \sum_{j=0}^\ell C_j y^{-m+\alpha/2-\beta}.
	  	\end{align*}
	  The last inequality is obtained by observing that $ y \le y+jt \le (1+\ell) t$ inside the integrals together with the  discussion  about the sign of $m-\alpha/2$.

	  	Finally we prove (ii). Assume that $\beta\in\N$ and $f\in 	\Lambda_H^{\alpha}. $
	  		  	Understanding now 
   	$(-\Delta_d)^\beta f$ as the $\beta$-times iteration  of $(-\Delta_d)$, and taking into account that $-\Delta_df(n)=\delta_{right}^2f(n-1)$, the result follows from applying $2\beta$ times Theorem \ref{derivada}.
   	
	  	\edproof

\bibliographystyle{siam}
\bibliography{references2}

\end{document}